\documentclass[preprint,3p,11pt,authoryear]{elsarticle}

\usepackage{amssymb,amsthm,mathtools}
\usepackage[colorlinks]{hyperref}

\usepackage{enumerate,natbib,caption,graphicx,subcaption}
\mathtoolsset{showonlyrefs=true} 
\usepackage{geometry}
\usepackage[title]{appendix} 
\usepackage{mathrsfs} 
\usepackage{dsfont} 
\usepackage{float} 

\newtheorem{theorem}{Theorem}
\newtheorem{corollary}{Corollary}
\newtheorem{lemma}{Lemma}
\newtheorem{remark}{Remark}
\newtheorem*{notations}{Notations}

\newcommand{\N}{\mathbb{N}}
\newcommand{\R}{\mathbb{R}}

\newcommand{\QQ}{\mathbb{Q}}
\newcommand{\PP}{\mathbb{P}}
\newcommand{\EE}{\mathbb{E}}

\newcommand{\OO}{\mathcal{O}}
\newcommand{\oo}{\mathrm{o}}
\newcommand{\leqdef}{\vcentcolon=}
\newcommand{\reqdef}{=\vcentcolon}
\newcommand{\rd}{{\rm d}}
\newcommand{\ind}{\mathds{1}}

\allowdisplaybreaks

\begin{document}

\begin{frontmatter}

    \title{A refined continuity correction for the negative binomial distribution and asymptotics of the median}%

    \author[a1,a2]{Fr\'ed\'eric Ouimet\texorpdfstring{}{)}}%

    \address[a1]{California Institute of Technology, Pasadena, USA.}%
    \address[a2]{Universit\'e de Montr\'eal, Montreal, QC H3T 1J4, Canada.}%

    \ead{frederic.ouimet@umontreal.ca}%

    \begin{abstract}
        In this paper, we prove a local limit theorem and a refined continuity correction for the negative binomial distribution.
        We present two applications of the results.
        First, we find the asymptotics of the median for a $\mathrm{Negative\hspace{0.5mm}Binomial}\hspace{0.2mm}(r,p)$ random variable jittered by a $\mathrm{Uniform}\hspace{0.2mm}(0,1)$, which answers a problem left open in \cite{MR4135709}. This is used to construct a simple, robust and consistent estimator of the parameter~$p$, when $r > 0$ is known. The case where $r$ is unknown is also briefly covered.
        Second, we find an upper bound on the Le Cam distance between negative binomial and normal experiments.
    \end{abstract}

    \begin{keyword}
        local limit theorem \sep continuity correction \sep quantile coupling \sep negative binomial distribution \sep Gaussian approximation \sep median \sep comparison of experiments, Le Cam distance, total variation
        \MSC[2020]{Primary: 62E20 Secondary: 62F12, 62F35, 62E15, 60F15, 62B15}
    \end{keyword}

\end{frontmatter}

\section{Introduction}\label{sec:intro}

    For any $r > 0$ and $p\in (0,1)$, the $\mathrm{Neg\hspace{0.3mm}Bin}\hspace{0.2mm}(r, p)$ probability mass function is defined by
    \vspace{-1mm}
    \begin{equation}\label{eq:negative.binomial.pmf}
        P_{r,p}(k) = \frac{\Gamma(r + k)}{\Gamma(r) \, k!} \, q^r p^k, \quad k\in \N_0,
    \end{equation}
    where $q\leqdef 1 - p$.
    If $r\in \N$, then a $\mathrm{Neg\hspace{0.3mm}Bin}\hspace{0.2mm}(r, p)$ random variable here is interpreted as the number of successes it takes to observe the $r$-th failure in an infinite sequence of independent trials.
    The first objective of our paper is to derive, using only elementary methods, a local asymptotic expansion for \eqref{eq:negative.binomial.pmf} in terms of the Gaussian density with the same mean and variance evaluated at $k$, namely:
    \begin{equation}\label{eq:phi.M}
        \frac{q}{\sqrt{r p}} \phi(\delta_k), \quad \text{where } \phi(z) \leqdef \frac{e^{-z^2/2}}{\sqrt{2\pi}} ~~ \text{and} ~~ \delta_k \leqdef \frac{k - r p q^{-1}}{\sqrt{r p q^{-2}}}.
    \end{equation}
    This is then used to prove a refined continuity correction for the negative binomial distribution in Section~\ref{sec:main.results}.
    For a general presentation on local limit theorems, we refer the reader to \cite{MR1295242}.

    In Section~\ref{sec:applications}, two applications of the local limit theorem and continuity correction are presented.
    Specifically, in Section~\ref{sec:median.jittered}, we find the asymptotics of the median for a $\mathrm{Neg\hspace{0.3mm}Bin}\hspace{0.2mm}(r,p)$ random variable jittered by a $\mathrm{Uniform}\hspace{0.2mm}(0,1)$, which answers a problem left open in Section 4 of \cite{MR4135709}. This result is used to construct a robust estimator of the parameter $p$ for the negative binomial distribution, which we briefly compare with the maximum likelihood estimator. The method of proof is completely new and might be of independent interest.
    To demonstrate its versatility, a slightly weaker version of the main result in \cite{MR4135709} is rederived with the new method in Appendix~\ref{sec:jittered.Poisson.median}. In Section~\ref{sec:LeCam.upper.bound}, we find an upper bound on the Le Cam distance between negative binomial and normal experiments. Some details on related works are presented in each subsection.

    \begin{notations}
        Throughout the paper, the notation $u = \OO(v)$ means that $\limsup |u / v| < C$, as $r\to \infty$, where $C > 0$ is a universal constant.
        Whenever $C$ might depend on some parameters, we add subscripts (for example, $u = \OO_p(v)$).
        Similarly, $u = \oo(v)$ means that $\lim |u / v| = 0$ as $r\to \infty$, and subscripts indicate which parameters the convergence rate can depend on.
        The notation $u \asymp v$ means that $u = \OO(v)$ and $v = \OO(u)$.
    \end{notations}

\section{A refined continuity correction}\label{sec:main.results}

    First, we need local approximations for the probability mass function of the negative binomial distribution with respect to the normal density function with the same mean and variance.

    \begin{lemma}[Local limit theorem]\label{lem:LLT.negative.binomial}
        For any $r > 0$ and $p,\eta\in (0,1)$, let
        \begin{equation}\label{eq:bulk}
            B_{r,p}(\eta) \leqdef \bigg\{k\in \N_0 : \Big|\frac{\delta_k}{\sqrt{r p}}\Big| \leq \eta \, r^{-1/3}\bigg\}
        \end{equation}
        denote the bulk of the negative binomial distribution.
        Then, as $r\to \infty$ and uniformly for $k\in B_{r,p}(\eta)$, we have
        \begin{align}\label{eq:lem:LLT.negative.binomial.log.eq}
            \log\bigg(\frac{P_{r,p}(k)}{\frac{q}{\sqrt{r p}} \phi(\delta_k)}\bigg)
            &= (r p)^{-1/2} \left\{\frac{1 + p}{6} \delta_k^3 - \frac{1 + p}{2} \delta_k\right\} \notag \\
            &\quad+ (r p)^{-1} \left\{- \frac{1 + p + p^2}{12} \delta_k^4 + \frac{p^2 + 1}{4} \delta_k^2 - \frac{p^2 + q}{12}\right\} \notag \\
            &\quad+ \OO_p\bigg(\frac{1 + |\delta_k|^5}{r^{3/2} \eta^4}\bigg),
        \end{align}
        and
        \begin{align}\label{eq:lem:LLT.negative.binomial.eq}
            \frac{P_{r,p}(k)}{\frac{q}{\sqrt{r p}} \phi(\delta_k)} = 1
            &+ (r p)^{-1/2} \left\{\frac{1 + p}{6} \delta_k^3 - \frac{1 + p}{2} \delta_k\right\} \notag \\
            &+ (r p)^{-1} \left\{\frac{(1 + p)^2}{72} \delta_k^6 - \frac{2 + 3 p + 2 p^2}{12} \delta_k^4 + \frac{3 + 2 p + 3 p^2}{8} \delta_k^2 - \frac{p^2 + q}{12}\right\} \notag \\
            &+ \OO_p\bigg(\frac{1 + |\delta_k|^9}{r^{3/2} \eta^4}\bigg).
        \end{align}
    \end{lemma}

    \begin{proof}
        The proof is differed to Appendix~\ref{sec:refined.continuity.correction}.
    \end{proof}

    \begin{remark}
        An expansion akin to \eqref{eq:lem:LLT.negative.binomial.eq} can be found in Result 4.3.3 of \cite{MR207011}. His paper states that, uniformly in $k$,
        \begin{equation}\label{eq:LLT.neg.bin.more.precise.1}
            \begin{aligned}
                P_{r,p}(k)
                &= \frac{q}{\sqrt{r p}} \phi(\delta_k) - \frac{q (1 + p)}{6 r p} \phi^{(3)}(\delta_k) \\[-1mm]
                &~+ \frac{q}{24 (r p)^{3/2}} \left\{(1 + 4 p + p^2) \phi^{(4)}(\delta_k) + \frac{(1 + p)^2}{3} \phi^{(6)}(\delta_k)\right\} \\[1mm]
                &~- \frac{q (1 + p)}{24 (r p)^2} \left\{\frac{1 + 10 p + p^2}{5} \phi^{(5)}(\delta_k) + \frac{1 + 4 p + p^2}{6} \phi^{(7)}(\delta_k) + \frac{(1 + p)^2}{54} \phi^{(9)}(\delta_k)\right\} \\[1mm]
                &~+ \oo_p\left((r p)^{-2}\right),
            \end{aligned}
        \end{equation}
        where $\phi^{(n)}(x) = (-1)^n \mathrm{He}_n(x) \phi(x)$, $\phi$ denotes the density function of the standard normal distribution, and $\mathrm{He}_n$ is the $n$-th probabilists' Hermite polynomial.
        The error term $\oo_p\left((r p)^{-2}\right)$ in \eqref{eq:LLT.neg.bin.more.precise.1} has the following form (see Theorem 5 in \cite{MR14626} or Result 3.8 in \cite{MR207011} for details):
        For any fixed $s \geq 6$,
        \begin{equation}\label{eq:LLT.neg.bin.more.precise.1.next.error}
            \oo_p\left((r p)^{-2}\right) = \frac{q}{\sqrt{r p}} \phi(\delta_k) \sum_{j=4}^{s-2} \frac{\text{Polynomial of order $3j$ in $\delta_k$}}{(r p)^{j/2}} + \oo_p((r p)^{-(s-1)/2}).
        \end{equation}
        If we try to divide both sides of \eqref{eq:LLT.neg.bin.more.precise.1} by $q \phi(\delta_k) / \sqrt{r p}$ in order to get an approximation for the ratio $P_{r,p}(k) / (q \phi(\delta_k) / \sqrt{r p})$, then the lingering error term on the right-hand side of \eqref{eq:LLT.neg.bin.more.precise.1.next.error} will become $\oo_p((r p)^{-(s-2)/2} e^{\delta_k^2 / 2})$, which gives a very poor control when $\delta_k$ is large.
        In fact, if we were to use this estimate for our application in Section~\ref{sec:LeCam.upper.bound} (look at Equations~\eqref{eq:I.plus.II.plus.III} and \eqref{eq:estimate.I.begin} in the proof of Theorem~\ref{thm:prelim.Carter}), we would need to control the error $\oo_p((r p)^{-(s-2)/2} \, \EE[e^{\delta_K^2 / 2}])$ as $r\to \infty$, where $K\sim \mathrm{Neg\hspace{0.3mm}Bin}\hspace{0.2mm}(r, p)$. But clearly, this cannot work since $\delta_K$ is roughly a standard normal when $r$ is large, say $Z$, and we have $\EE[e^{Z^2 / 2}] = \infty$.

        In that sense, the result of Lemma~\ref{lem:LLT.negative.binomial} gives a refined control of the last error term in the expansion of the ratio $P_{r,p}(k) / (q \phi(\delta_k) / \sqrt{r p})$ compared to \cite{MR207011}. In turn, the continuity correction based on that refined local limit theorem will be called a refined continuity correction, an expression originally coined by \cite{MR538319} when he proved similar results in the context of the binomial distribution.

        We should also mention that in contrast with Govindarajulu, our proof is elementary and self-contained; we only use Stirling's formula and Taylor expansions. Result 4.3.3 in \cite{MR207011} relies on difficult estimates of Fourier analysis found in \cite{MR14626}, which itself relies on results from \cite{doi:10.1080/03461238.1928.10416862,Cramer_1937_book}. This makes the result of Lemma~\ref{lem:LLT.negative.binomial} and its consequences in Section~\ref{sec:applications} more accessible and open to scrutiny.
    \end{remark}

    By summing up the local approximations in Lemma~\ref{lem:LLT.negative.binomial}, we can prove a (refined) continuity correction for the negative binomial distribution.
    An analogous result was proved for the binomial distribution in Theorem 2 of \cite{MR538319}.

    \begin{theorem}[Refined continuity correction]\label{thm:main.result}
        For any $r > 0$ and $p,\eta\in (0,1)$, recall the definition of $B_{r,p}(\eta)$ from \eqref{eq:bulk}.
        Then, as $r\to \infty$ and uniformly for $a\in B_{r,p}(\eta)$, we have
        \vspace{-1mm}
        \begin{align}
            &\sum_{k=a}^{\infty} P_{r,p}(k) = \Psi(\delta_{a - c_{r,p}^{\star}(a)}) + \OO_p(r^{-3/2}), \label{eq:thm:main.result.eq.1} \\
            &\sum_{k=0}^a P_{r,p}(k) = \Phi(\delta_{a + 1 - c_{r,p}^{\star}(a + 1)}) + \OO_p(r^{-3/2}), \label{eq:thm:main.result.eq.2.reflect}
        \end{align}
        where $\Psi = 1 - \Phi$ and $\Phi$ denotes the cumulative distribution function of the standard normal distribution, and the continuity correction is given by
        \begin{equation}\label{eq:optimal.choice.c}
            \begin{aligned}
                c_{r,p}^{\star}(a)
                &\leqdef \frac{1}{2} + \frac{1+p}{6 q} \big[\delta_{a - \frac{1}{2}}^2 - 1\big] \\
                &\qquad+ \frac{1}{q \sqrt{r p}}\left\{- \frac{1}{72} \big[5 + 16 p + 17 p^2\big] \delta_{a - \frac{1}{2}}^3 + \frac{1}{36} \big[1 - 4 p - 2 p^2\big] \delta_{a - \frac{1}{2}}\right\}.
            \end{aligned}
        \end{equation}
    \end{theorem}

    \begin{proof}
        The proof is differed to Appendix~\ref{sec:refined.continuity.correction}.
    \end{proof}

\section{Applications}\label{sec:applications}

    In this section, we present two applications of Lemma~\ref{lem:LLT.negative.binomial} and Theorem~\ref{thm:main.result}.

    \subsection{Median of a jittered negative binomial random variable}\label{sec:median.jittered}

        Incremental improvements on the bounds for the median of the Poisson distribution and the related asymptotics of the median for the Gamma distribution have a long history in probability and statistics, see, e.g., \cite{MR67384,MR448635,MR858317,MR1195477,MR1333373,MR1997033,MR2230380,MR2126774,MR2173334,MR2245572,MR2275872,MR2386240,MR2379680,MR2555416,MR2904370,MR3696135,MR3597404,doi:10.1371/journal.pone.0251626,MR4202215}.
        The best account of the more important developments is given in Section 1 of \cite{MR2386240}.
        For instance, the difference between the median and the mean of the $\mathrm{Poisson}\hspace{0.2mm}(\lambda)$ distribution fluctuates, as $\lambda\to \infty$, between two sharp bounds. This is something that we know since \cite{MR1195477}, who proved the following conjecture originally made by \cite{MR858317}: if $N_{\lambda}\sim \mathrm{Poisson}\hspace{0.2mm}(\lambda)$, then
        \begin{equation}\label{eq:sharp.bounds.Poisson.median}
            - \log 2 \leq \mathrm{Median}(N_{\lambda}) - \lambda < \frac{1}{3},
        \end{equation}
        and these bounds are the best possible.
        As a complement, it was proved in \cite{MR2230380} that
        \begin{equation}\label{eq:Poisson.median.fluctuates}
            \liminf_{\lambda\to \infty} \big\{\mathrm{Median}(N_{\lambda}) - \lambda\big\} = - \frac{2}{3} \quad \text{and} \quad \limsup_{\lambda\to \infty} \big\{\mathrm{Median}(N_{\lambda}) - \lambda\big\} = \frac{1}{3}.
        \end{equation}
        In particular, there is no convergence for the sequence $\{\mathrm{Median}(N_{\lambda}) - \lambda\}_{\lambda > 0}$, which is a property that stems from the discrete nature of the Poisson distribution.

        \vspace{2mm}
        Finding bounds for the median of the negative binomial distribution and developing approximations for the median of the beta distribution is a similar problem of interest, but it is more difficult.
        As such, it has been addressed by fewer authors.
        The first paper to provide such bounds is \cite{MR1050698}, who followed a method analogous to \cite{MR858317}, where bounds for the difference between the median and the mean of the Gamma distribution were found first and a Poisson-Gamma relation was exploited to deduce bounds for the median of the Poisson distribution. In a similar fashion, \cite{MR1050698} found bounds for the median of the beta distribution and exploited a relationship between the cumulative distribution functions of the beta and negative binomial distributions, first proved by \cite{10.1093/biomet/25.1-2.158} and rediscovered by \cite{MR119280}, to find bounds for the median of the negative binomial distribution.
        In \cite{MR1235080}, the authors explained that the definition of median in \cite{MR1050698} was not standard and proceeded to derive bounds for the median of the negative binomial distribution (following a similar strategy) using the standard definition instead, namely: $\mathrm{Median}(X) \leqdef \inf\{x\in \R : \PP(X \leq x) \geq \frac{1}{2}\}$. The same bounds were rederived using a different method and parametrization (as well as bounds on the  50 percentage point) in \cite{MR1278595}.

        \vspace{2mm}
        Recently, a different approach was proposed in \cite{MR4135709} to bypass the fact that there is no asymptotic limit, because of \eqref{eq:Poisson.median.fluctuates}, for the median of the Poisson distribution. Their idea was to jitter the Poisson random variable $N_{\lambda}$ by a $U\sim \mathrm{Uniform}\hspace{0.2mm}(0,1)$ random variable and then investigate the limit.
        They found the very nice result:
        \begin{equation}\label{eq:asymptotics.jittered.Poisson.median}
            \mathrm{Median}(N_{\lambda} + U) - \lambda = \frac{1}{3} + \oo(\lambda^{-1}).
        \end{equation}
        The analogous problem for the negative binomial was left open Section~4 of \cite{MR4135709}, which we solve in Theorem~\ref{thm:median.negative.binomial} below.
        Our method of proof is completely new and based on the continuity correction from Theorem~\ref{thm:main.result}.
        As a testament to the versatility of our method, we give a short alternative proof of a slightly weaker version of \eqref{eq:asymptotics.jittered.Poisson.median} in Appendix~\ref{sec:jittered.Poisson.median}, which might be of independent interest.

        \begin{theorem}\label{thm:median.negative.binomial}
            For $r > 0$ and $p\in (0,1)$, let $K\sim \mathrm{Neg\hspace{0.3mm}Bin}\hspace{0.2mm}(r, p)$ and $U\sim \text{Uniform}\hspace{0.2mm}(0,1)$.
            Then, as $r\to \infty$, we have
            \begin{equation}\label{eq:thm:median.negative.binomial}
                \mathrm{Median}\hspace{0.2mm}(K + U) - \frac{r p}{q} = \frac{1}{2} - \frac{1 + p}{6 q} + \OO_p(r^{-1}).
            \end{equation}
            This result is illustrated in Figure~\ref{fig:Figure.1.analogue.Coeurjolly} for multiple values of $p$.
        \end{theorem}

        \begin{figure}[ht]
            \captionsetup{width=0.9\linewidth}
            \centering
            \begin{subfigure}[b]{0.3\textwidth}
                \centering
                \includegraphics[width=\textwidth, height=0.85\textwidth]{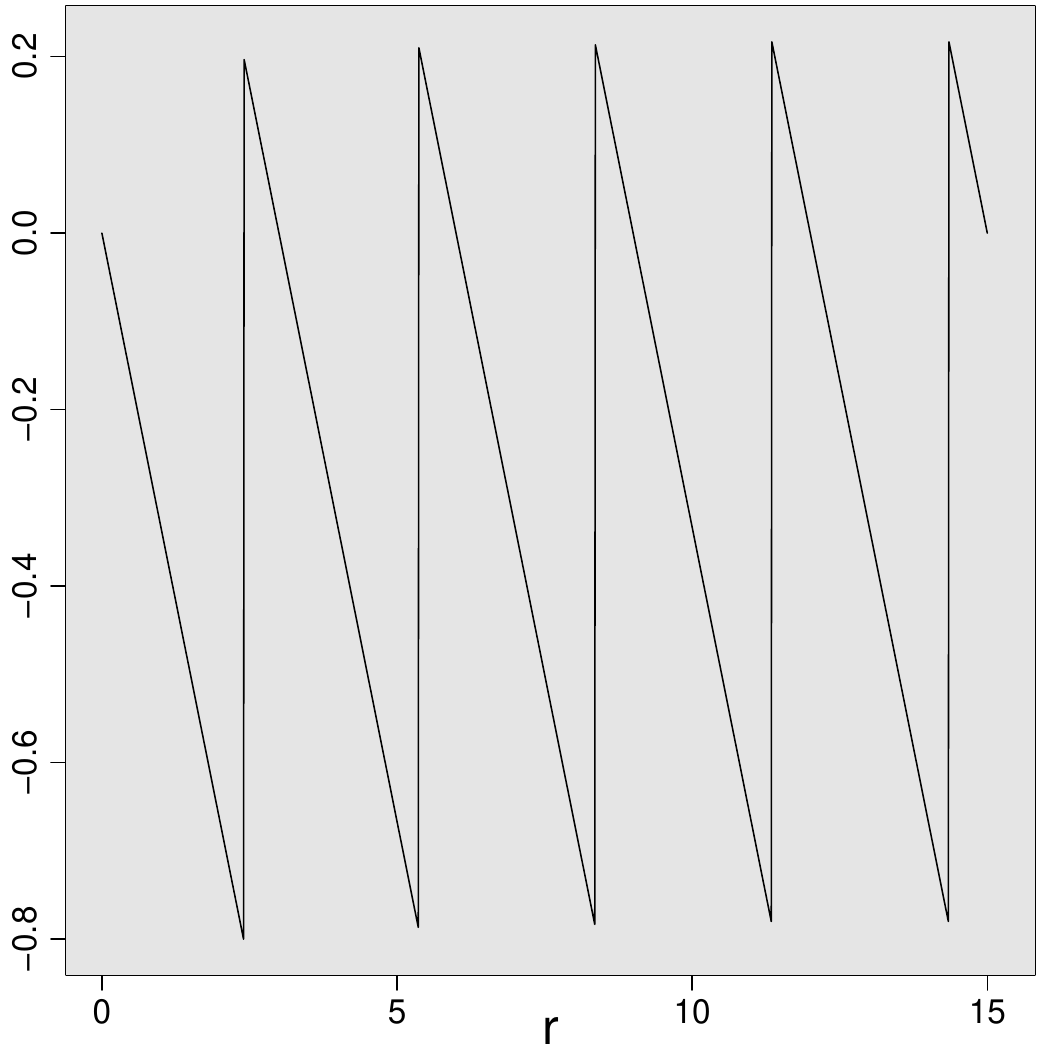}
                \vspace{-0.5cm}
                \caption{$p = 1/4$}
                \vspace{2mm}
            \end{subfigure}
            \quad
            \begin{subfigure}[b]{0.3\textwidth}
                \centering
                \includegraphics[width=\textwidth, height=0.85\textwidth]{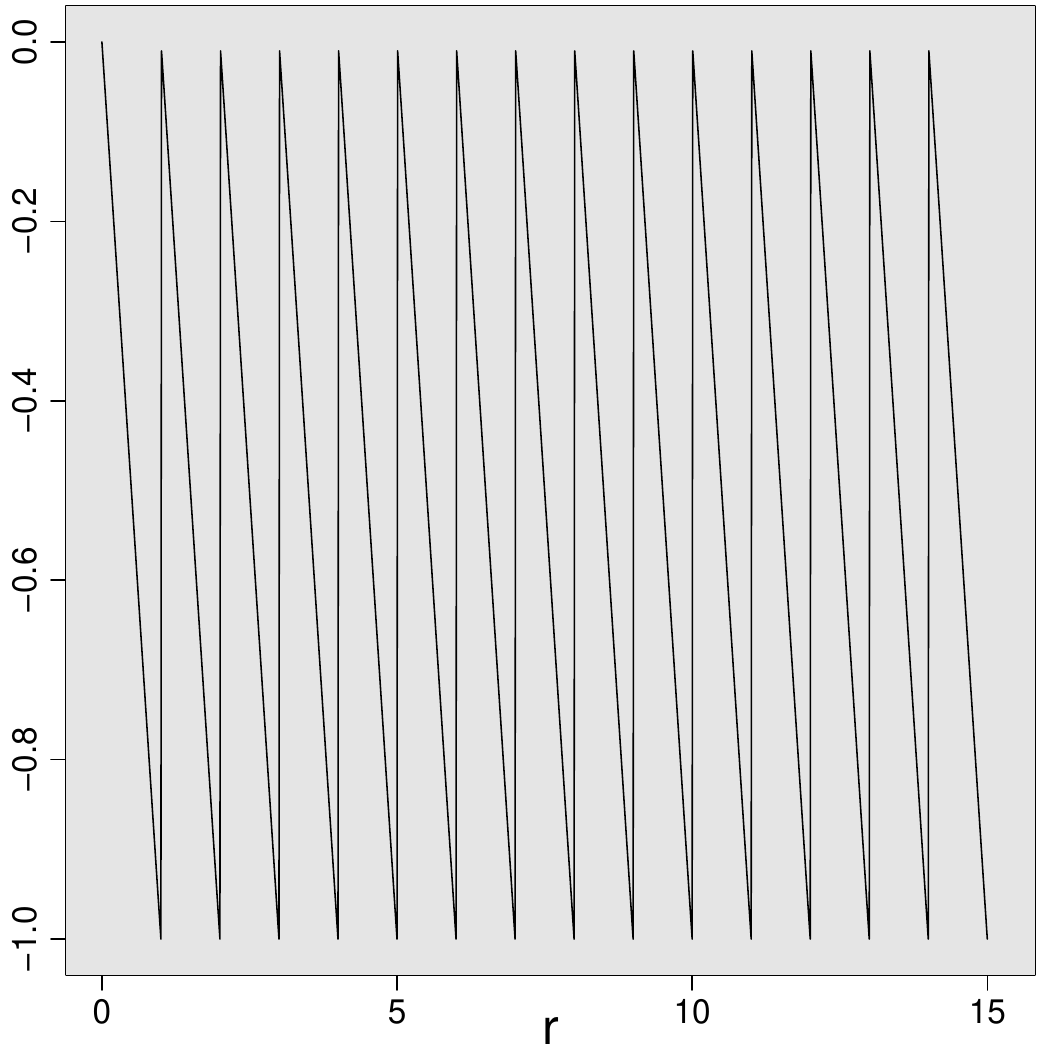}
                \vspace{-0.5cm}
                \caption{$p = 1/2$}
                \vspace{2mm}
            \end{subfigure}
            \quad
            \begin{subfigure}[b]{0.3\textwidth}
                \centering
                \includegraphics[width=\textwidth, height=0.85\textwidth]{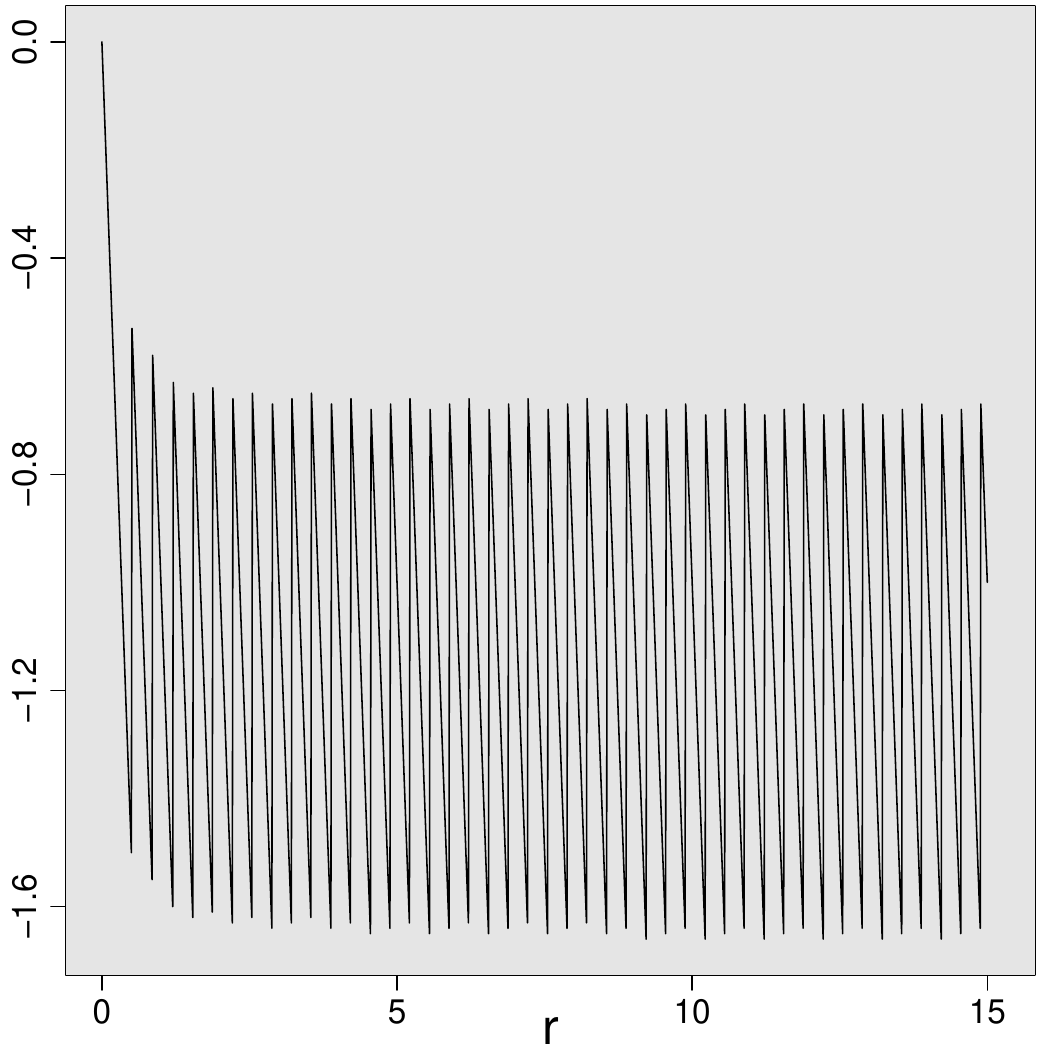}
                \vspace{-0.5cm}
                \caption{$p = 3/4$}
                \vspace{2mm}
            \end{subfigure}
            \begin{subfigure}[b]{0.3\textwidth}
                \centering
                \includegraphics[width=\textwidth, height=0.85\textwidth]{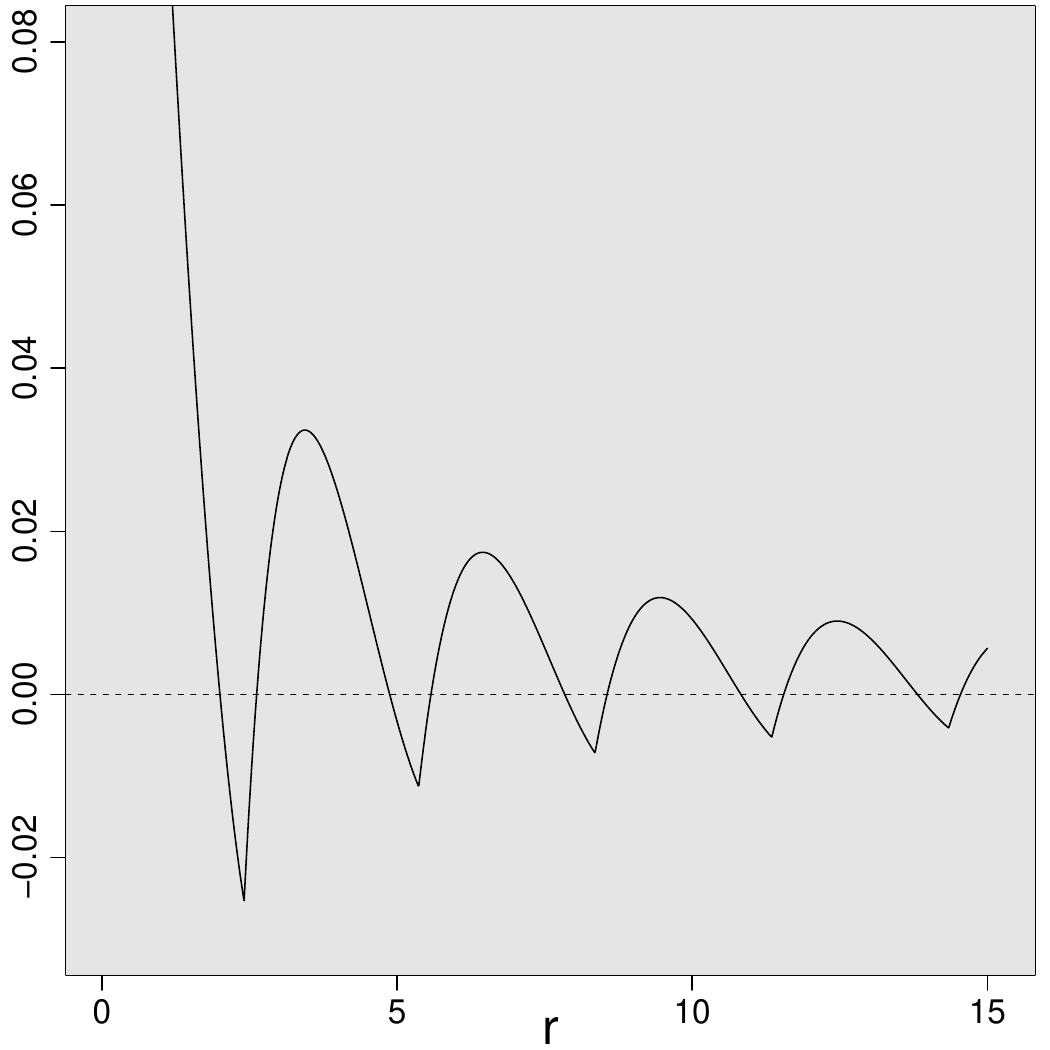}
                \vspace{-0.563cm}
                \caption{$p = 1/4$}
                \vspace{2mm}
            \end{subfigure}
            \quad
            \begin{subfigure}[b]{0.3\textwidth}
                \centering
                \includegraphics[width=\textwidth, height=0.85\textwidth]{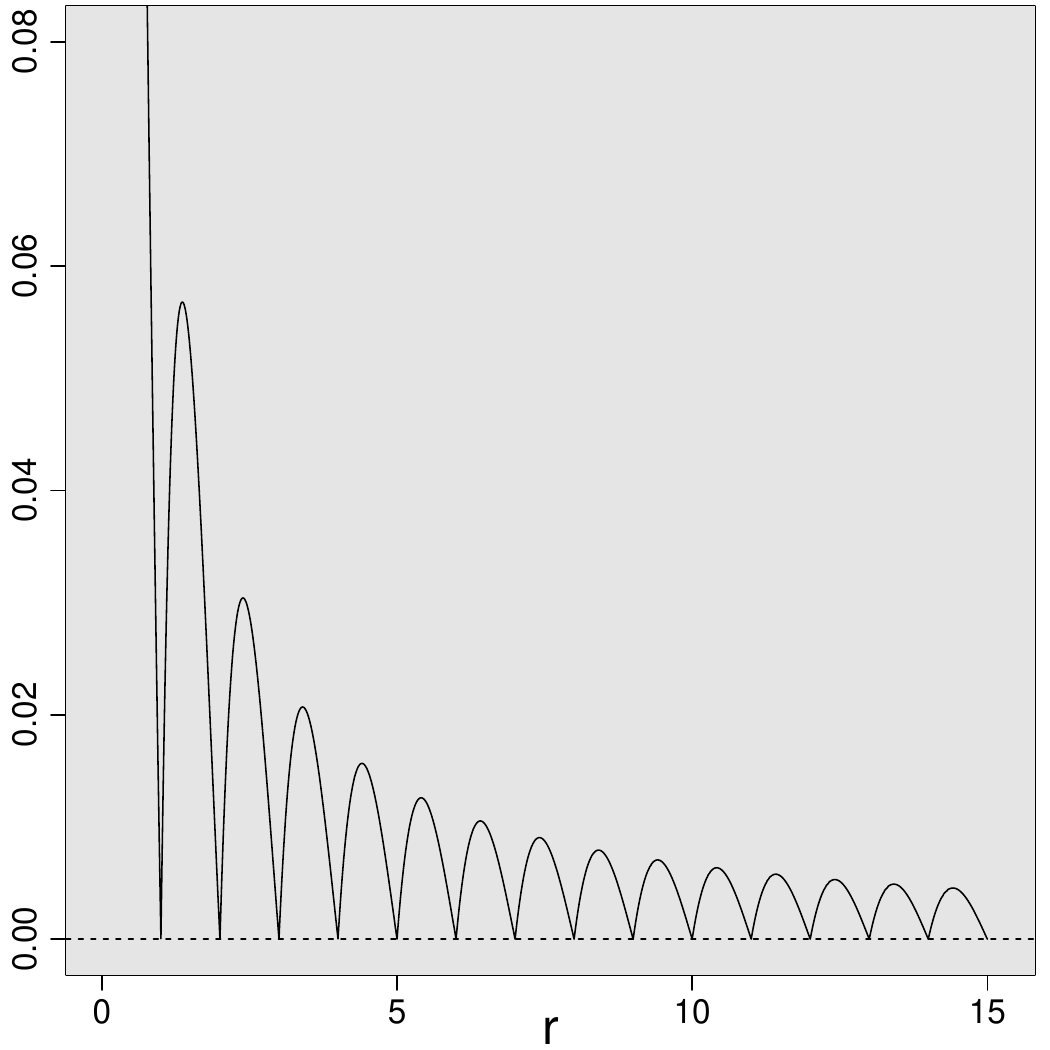}
                \vspace{-0.563cm}
                \caption{$p = 1/2$}
                \vspace{2mm}
            \end{subfigure}
            \quad
            \begin{subfigure}[b]{0.3\textwidth}
                \centering
                \includegraphics[width=\textwidth, height=0.85\textwidth]{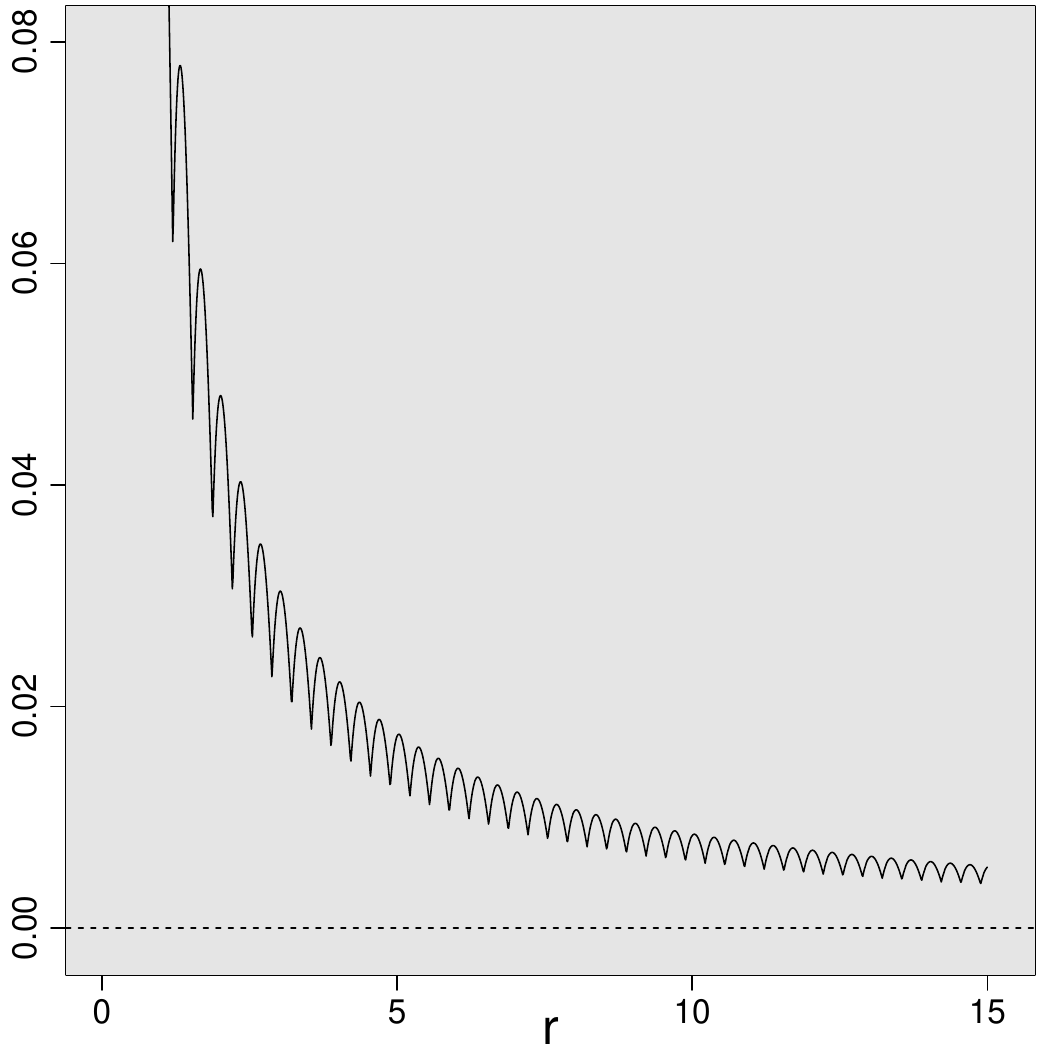}
                \vspace{-0.563cm}
                \caption{$p = 3/4$}
                \vspace{2mm}
            \end{subfigure}
            \vspace{-3mm}
            \caption{Illustration of $\mathrm{Med}(K) - \frac{r p}{q}$ (top row) and $\mathrm{Med}(K + U) - \frac{r p}{q} - (\frac{1}{2} - \frac{1 + p}{6 q})$ (bottom row), for $p\in \{\tfrac{1}{4},\tfrac{1}{2},\tfrac{3}{4}\}$.}
            \label{fig:Figure.1.analogue.Coeurjolly}
        \end{figure}

        \begin{proof}[Proof of Theorem~\ref{thm:median.negative.binomial}]
            By conditioning on $U$ and using the continuity correction from Theorem~\ref{thm:main.result}, we want to find $t = \mathrm{Median}\hspace{0.2mm}(K + U) > 0$ such that
            \begin{align}\label{eq:thm:median.negative.binomial.beginning}
                \frac{1}{2}
                &= \int_0^1 \PP(K \leq t - u) \, \rd u \notag \\[0.5mm]
                &= \PP(K \leq \lfloor t \rfloor) \cdot \{t\} + \PP(K \leq \lfloor t \rfloor - 1) \cdot (1 - \{t\}) \notag \\[2mm]
                &= \Phi(\delta_{\lfloor t \rfloor + 1 - c_{r,p}^{\star}(\lfloor t \rfloor + 1)}) \cdot \{t\} + \Phi(\delta_{\lfloor t \rfloor - c_{r,p}^{\star}(\lfloor t \rfloor)}) \cdot (1 - \{t\}) + \OO_p(r^{-3/2}),
            \end{align}
            where $\{t\}$ denotes the fractional part of $t$.
            Now, we have the following Taylor series expansion for $\Phi$ at $0$:
            \begin{equation}
                \Phi(x) = \frac{1}{2} + \frac{x}{\sqrt{2\pi}} + \OO(x^3).
            \end{equation}
            Therefore, \eqref{eq:thm:median.negative.binomial.beginning} becomes
            \begin{equation}\label{eq:upto}
                0 = \frac{1}{\sqrt{2\pi}} \left[\delta_{\lfloor t \rfloor + 1 - c_{r,p}^{\star}(\lfloor t \rfloor + 1)} \cdot \{t\} + \delta_{\lfloor t \rfloor - c_{r,p}^{\star}(\lfloor t \rfloor)} \cdot (1 - \{t\})\right] + \OO_p(r^{-3/2}).
            \end{equation}
            After rearranging some terms, this is equivalent to
            \begin{equation}
                 t - \frac{r p}{q} = c_{r,p}^{\star}(\lfloor t \rfloor + 1) \cdot \{t\} + c_{r,p}^{\star}(\lfloor t \rfloor) \cdot (1 - \{t\}) + \OO_p(r^{-1}).
            \end{equation}
            By applying the expression for $c_{r,p}^{\star}$ in \eqref{eq:optimal.choice.c}, this is
            \begin{equation}
                 t - \frac{r p}{q} = \frac{1}{2} - \frac{1 + p}{6 q} + \OO_p(r^{-1}).
            \end{equation}
            This ends the proof.
        \end{proof}

        When $r$ is known, we can follow the main idea in \cite{MR4135709} to build a simple, robust and consistent estimator for the parameter $p$ of the $\mathrm{Neg\hspace{0.3mm}Bin}\hspace{0.2mm}(r,p)$ distribution. The estimator is robust to outliers in the sense that it depends only on the sample median.

        \begin{corollary}[Robust estimator for $p$ when $r$ is known]\label{cor:robust.estimator}
            Let $r > 0$ and $p\in (0,1)$.
            Let $K_1,K_2,\dots,K_n\sim \mathrm{Neg\hspace{0.3mm}Bin}\hspace{0.2mm}(r,p)$ and $U_1,U_2,\dots,U_n\sim \mathrm{Uniform}\hspace{0.2mm}(0,1)$ be i.i.d., and define $X_i \leqdef K_i + U_i$ for all $i\in \{1,2,\dots,n\}$.
            Then,
            \begin{equation}\label{eq:cor:robust.estimator}
                \hat{p}_n^{\scriptscriptstyle(\mathrm{R})} \leqdef \frac{\widehat{\mathrm{Median}}(X_1, X_2, \dots, X_n) - \frac{1}{3}}{\widehat{\mathrm{Median}}(X_1, X_2, \dots, X_n) - \frac{2}{3} + r} \stackrel{\PP}{\longrightarrow} p,
            \end{equation}
            as $n\to \infty$ and then $r\to \infty$.
        \end{corollary}

        \begin{proof}[Proof of Corollary~\ref{cor:robust.estimator}]
            From standard asymptotic theory, see, e.g., \cite{MR1652247}, p.47, we have the convergence in probability of the sample median:
            \begin{equation}
                \widehat{\mathrm{Median}}(X_1, X_2, \dots, X_n) \stackrel{\PP}{\longrightarrow} \mathrm{Median}\hspace{0.2mm}(X_1), \quad \text{as } n\to \infty.
            \end{equation}
            Hence, by Theorem~\ref{thm:median.negative.binomial}, we have, as $n\to \infty$ and then $r\to \infty$,
            \begin{equation}\label{eq:cor:robust.estimator.last}
                \widehat{\mathrm{Median}}(X_1, X_2, \dots, X_n) - \frac{6 r p - (1 + p)}{6 (1 - p)}\stackrel{\PP}{\longrightarrow} \frac{1}{2}.
            \end{equation}
            By equalizing both sides and solving for $p$, we get the estimator $\hat{p}_n^{\scriptscriptstyle(\mathrm{R})}$ in \eqref{eq:cor:robust.estimator}, which converges in probability to $p$ by \eqref{eq:cor:robust.estimator.last} and the continuous mapping theorem.
        \end{proof}

        In the top row of Figure~\ref{fig:Figure.2.3.analogue.Coeurjolly} below, a comparison between the empirical biases of the robust estimator, $\hat{p}_n^{\scriptscriptstyle(\mathrm{R})}$, and the empirical biases of the maximum likelihood (ML) estimator, $\hat{p}_n^{\scriptscriptstyle(\mathrm{ML})}$, is made.
        The ratio of their root mean squared error (RMSE) is presented in the bottom row of the same figure.
        We see that the greater robustness to outliers of the estimator $\hat{p}_n^{\scriptscriptstyle(\mathrm{R})}$, compared to $\hat{p}_n^{\scriptscriptstyle(\mathrm{ML})} = (1 + r \hspace{0.2mm}n / \sum_{i=1}^n K_i)^{-1}$, comes at the cost of a slightly worst performance.

        \vspace{2mm}
        \begin{figure}[ht]
            \captionsetup{width=0.90\linewidth}
            \centering
            \begin{subfigure}[b]{0.400\textwidth}
                \centering
                \includegraphics[width=\textwidth, height=0.85\textwidth]{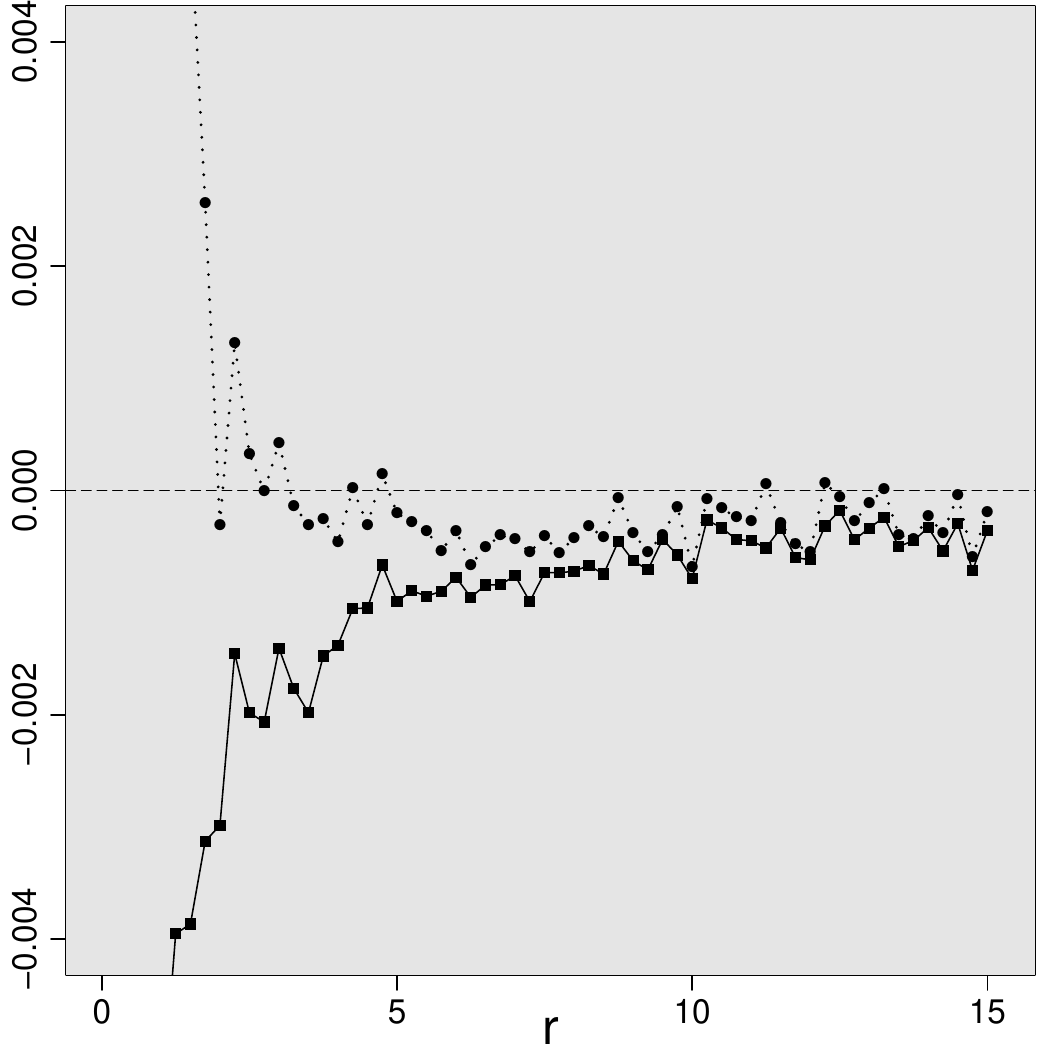}
                \vspace{-0.5cm}
                \caption{$n = 50$.\vspace{2mm}}
            \end{subfigure}
            \quad
            \begin{subfigure}[b]{0.400\textwidth}
                \centering
                \includegraphics[width=\textwidth, height=0.85\textwidth]{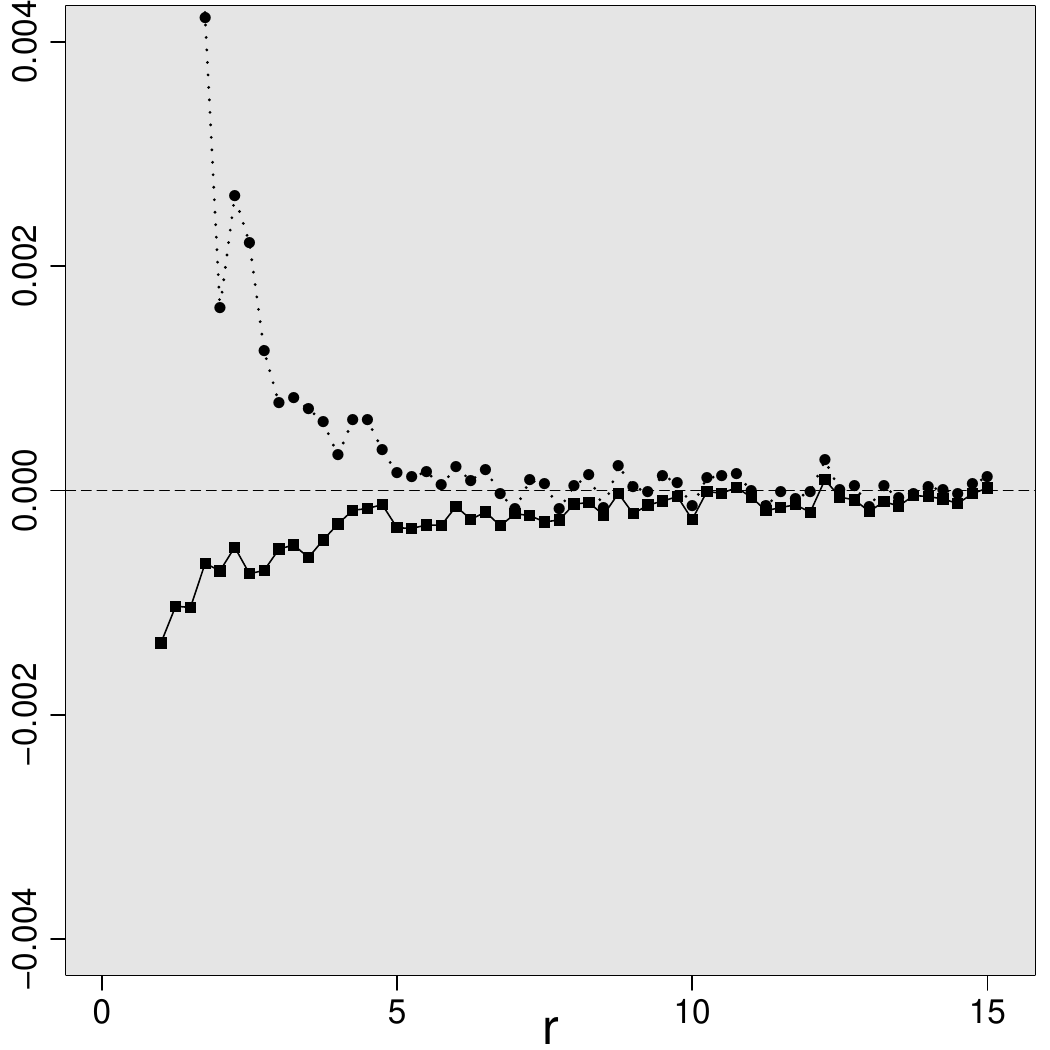}
                \vspace{-0.563cm}
                \caption{$n = 200$.\vspace{2mm}}
            \end{subfigure} \\
            \begin{subfigure}[b]{0.400\textwidth}
                \centering
                \includegraphics[width=\textwidth, height=0.85\textwidth]{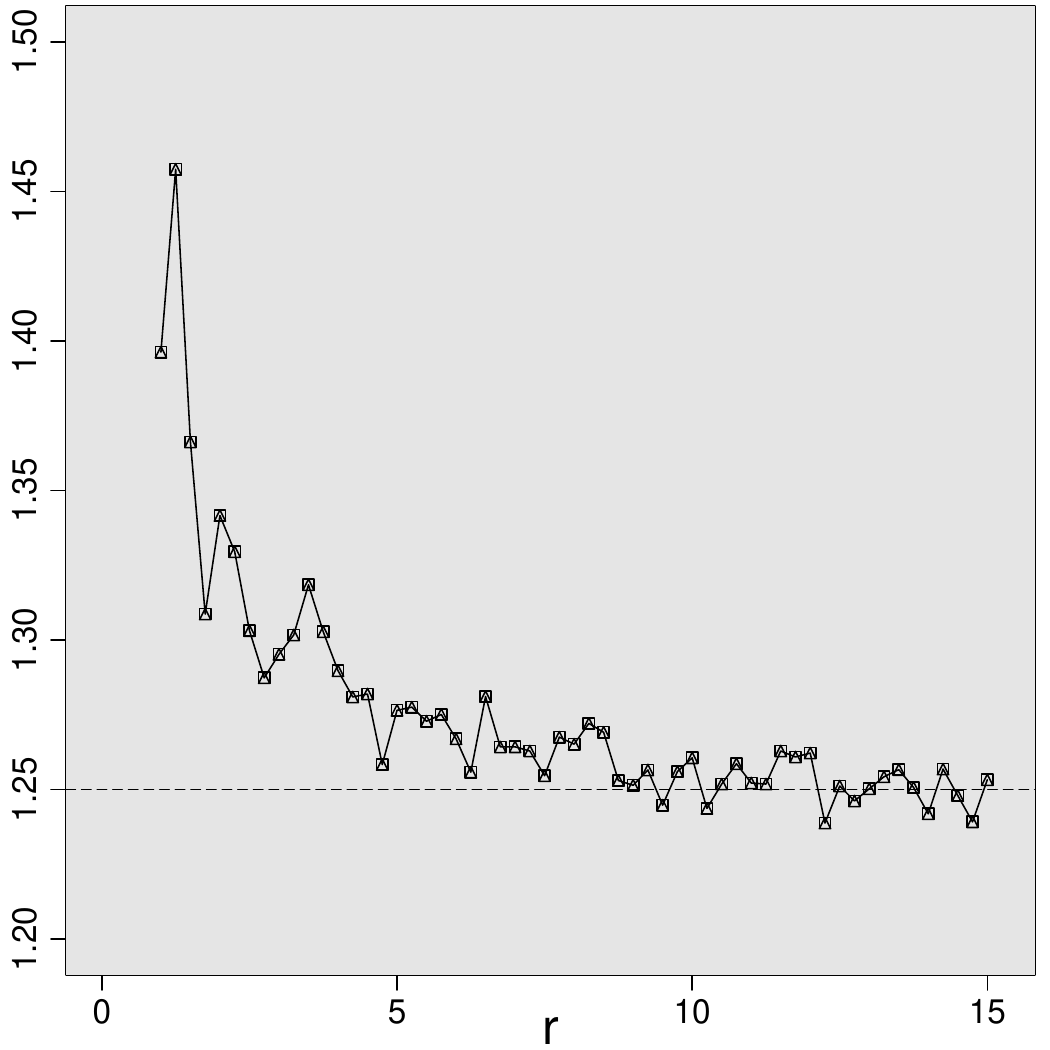}
                \vspace{-0.5cm}
                \caption{$n = 50$.}
            \end{subfigure}
            \quad
            \begin{subfigure}[b]{0.400\textwidth}
                \centering
                \includegraphics[width=\textwidth, height=0.85\textwidth]{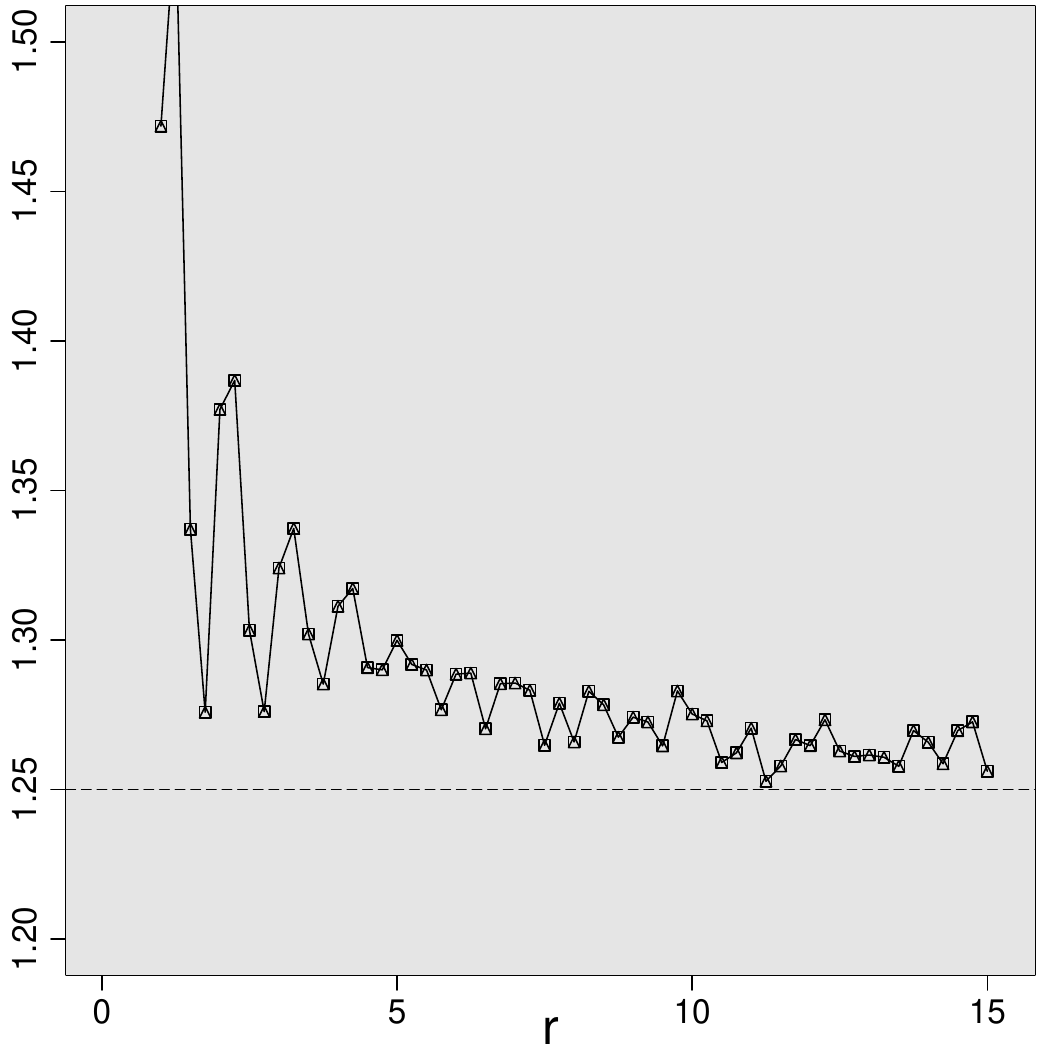}
                \vspace{-0.563cm}
                \caption{$n = 200$.}
            \end{subfigure}
            \vspace{-2mm}
            \caption{The top row shows the empirical biases for the estimators of $p$ (squares and solid lines for $\hat{p}_n^{\scriptscriptstyle(\mathrm{ML})}$, and disks and dotted lines for $\hat{p}_n^{\scriptscriptstyle(\mathrm{R})}$), based on 10,000 samples of size $n$ from a $\mathrm{Neg\hspace{0.3mm}Bin}\hspace{0.2mm}(r,p)$ distribution, where $p = 1/2$ and $r\in \frac{1}{4} \cdot \{2,3,\dots,60\}$. The bottom row shows the RMSE for the robust estimator, $\hat{p}_n^{\scriptscriptstyle(\mathrm{R})}$, over the ML estimator, $\hat{p}_n^{\scriptscriptstyle(\mathrm{ML})}$.}
            \label{fig:Figure.2.3.analogue.Coeurjolly}
        \end{figure}

        However, when $1\%$ of the data is polluted by a hundred times scaled Poisson random variables (with mean $\lambda = r p$), then Figure~\ref{fig:Figure.2.3.analogue.Coeurjolly} changes to Figure~\ref{fig:Figure.2.3.analogue.Coeurjolly.polluted} below.
        The empirical biases of the ML estimator become significant whereas our robust estimator remains unaffected. This was expected because our robust estimator depends only on the median, and the sample median is unaltered by (extreme) outliers. The RMSE hovering around 0.25 now heavily favors our robust estimator even if a very small fraction of the data was polluted. This means that the previous RMSE of 1.25 can be a small price to pay in performance in practice if we are not sure that the data we are concerned with is purely negative binomial due to some observations deviating significantly from the bulk. When the Poisson pollutants are only scaled by 10 instead of 100 (simulations are not shown here), the RMSE is below $1$ for all $r \geq 7$ (i.e., in favor of the robust estimator). If the scale factor stays at $10$ but the pollution rate increases above $1\%$, than the whole RMSE curve is pushed downward below $1$ the further the pollution rate increases, again favoring the robust estimator. If the scale factor of the Poisson pollutant is close to $1$ however, then the ML estimator has the advantage (a RMSE around 1.25). Therefore, care should be taken before choosing the robust estimator when there are no clear outliers in the data. If there seems to be multiple (severe) outliers in the data however, than the robust estimator is a safe bet compared to the ML estimator and should perform better overall.

        \vspace{2mm}
        \begin{figure}[ht]
            \captionsetup{width=0.90\linewidth}
            \centering
            \begin{subfigure}[b]{0.400\textwidth}
                \centering
                \includegraphics[width=\textwidth, height=0.85\textwidth]{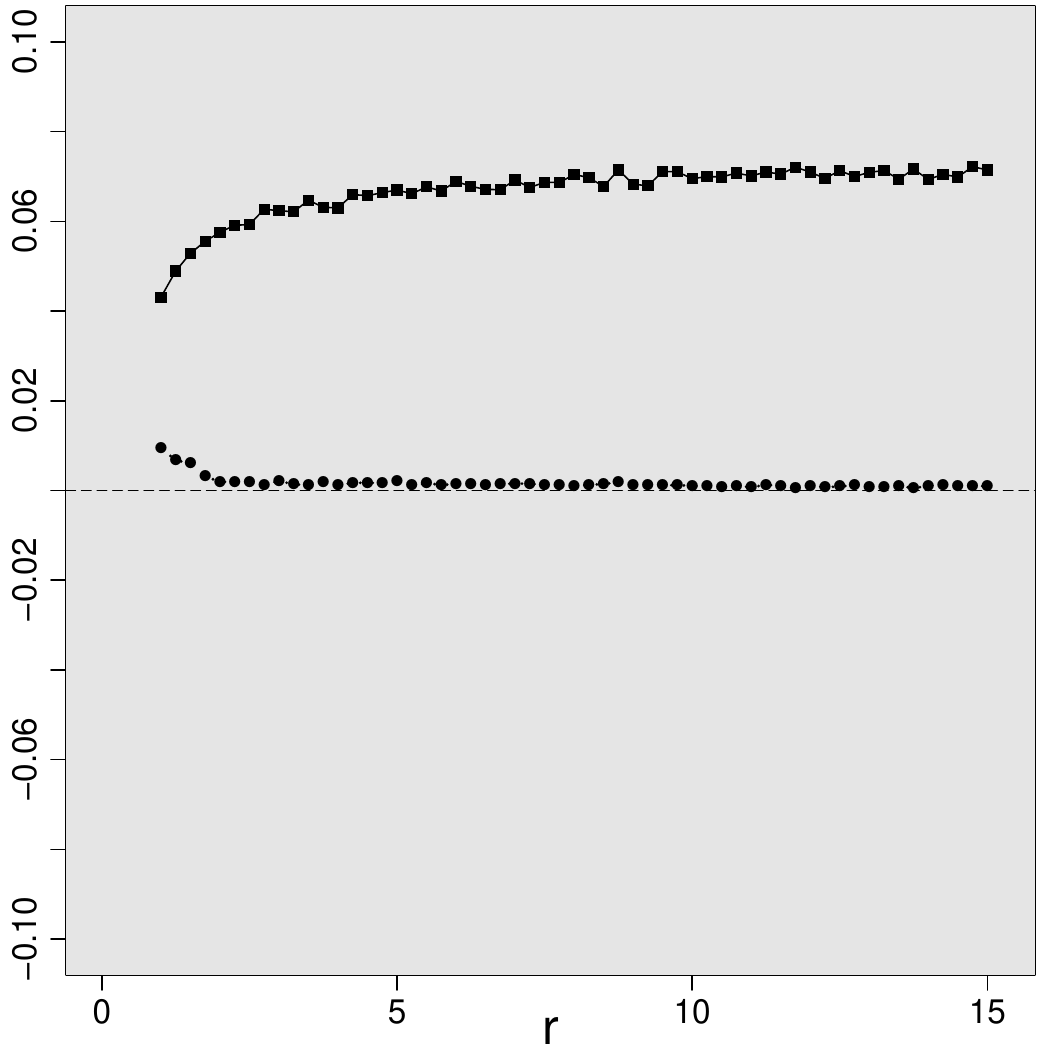}
                \vspace{-0.5cm}
                \caption{$n = 50$.\vspace{2mm}}
            \end{subfigure}
            \quad
            \begin{subfigure}[b]{0.400\textwidth}
                \centering
                \includegraphics[width=\textwidth, height=0.85\textwidth]{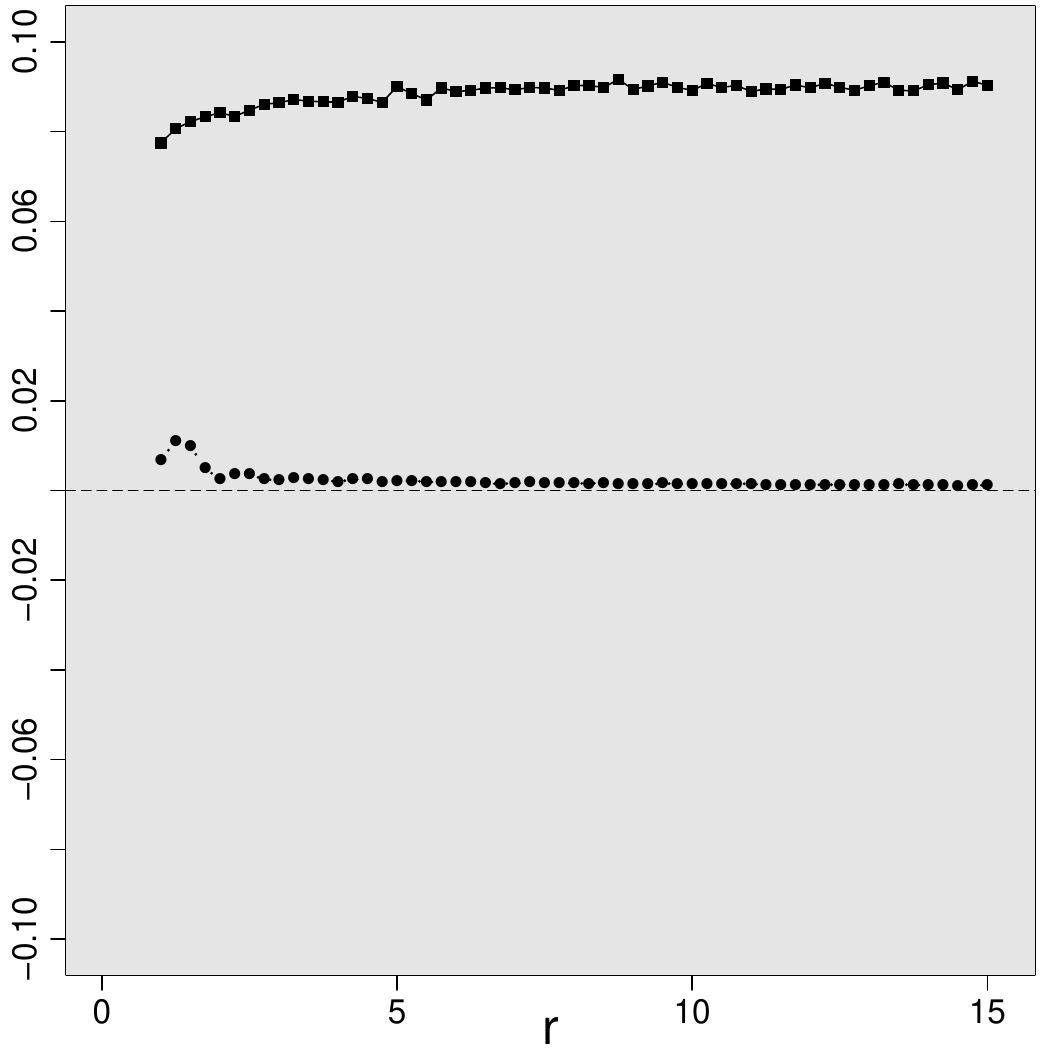}
                \vspace{-0.563cm}
                \caption{$n = 200$.\vspace{2mm}}
            \end{subfigure}
            \begin{subfigure}[b]{0.400\textwidth}
                \centering
                \includegraphics[width=\textwidth, height=0.85\textwidth]{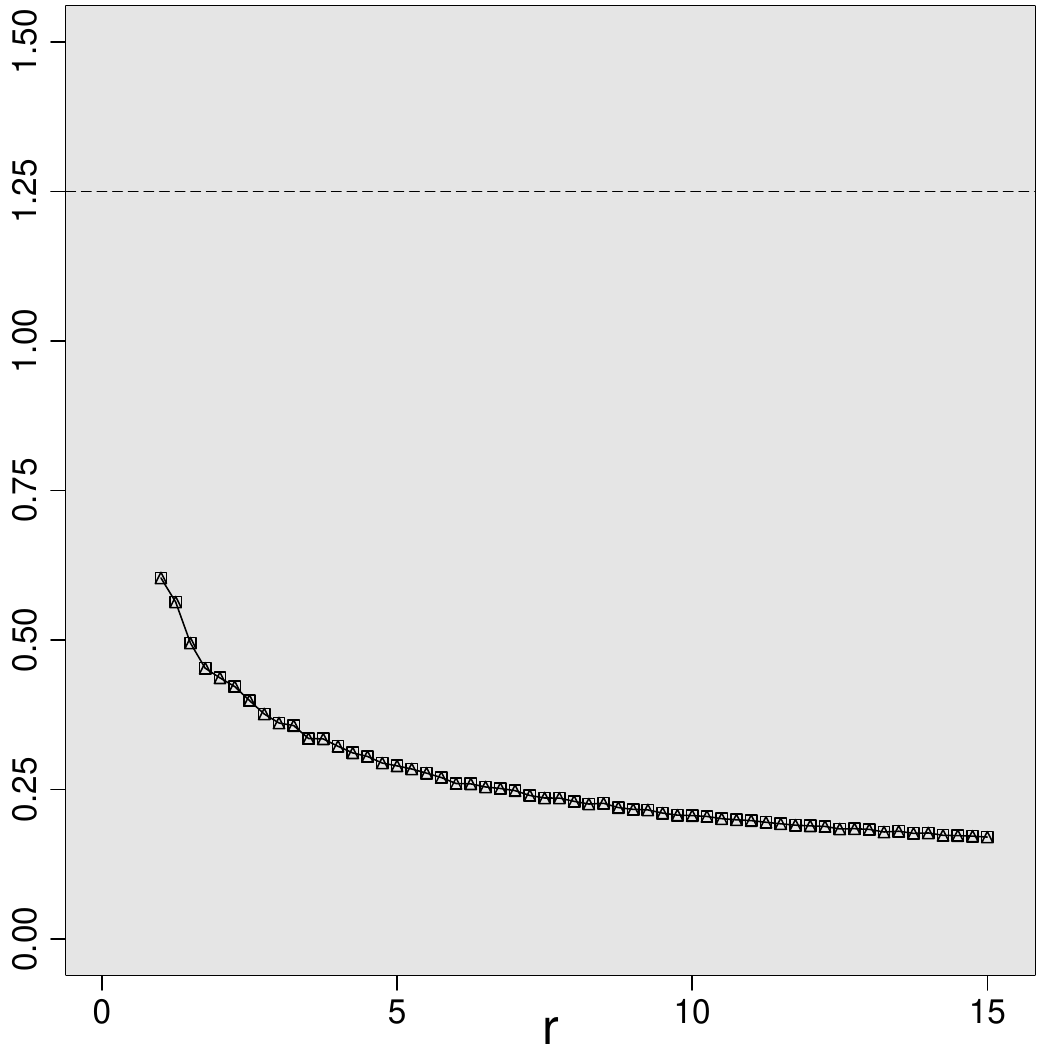}
                \vspace{-0.5cm}
                \caption{$n = 50$.}
            \end{subfigure}
            \quad
            \begin{subfigure}[b]{0.400\textwidth}
                \centering
                \includegraphics[width=\textwidth, height=0.85\textwidth]{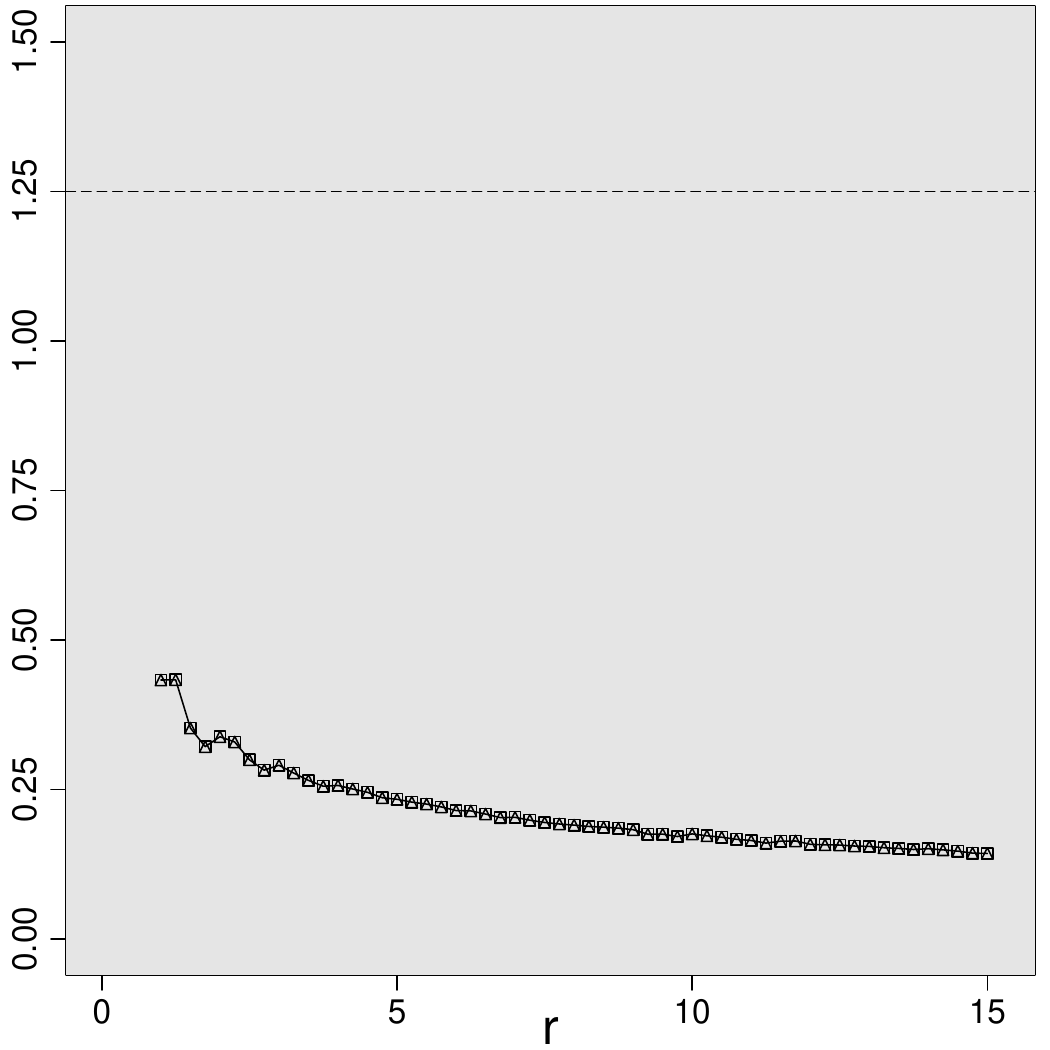}
                \vspace{-0.563cm}
                \caption{$n = 200$.}
            \end{subfigure}
            \vspace{-2mm}
            \caption{The top row shows the empirical biases for the estimators of $p$ (squares and solid lines for $\hat{p}_n^{\scriptscriptstyle(\mathrm{ML})}$, and disks and dotted lines for $\hat{p}_n^{\scriptscriptstyle(\mathrm{R})}$), based on 10,000 samples of size $n$ from a polluted negative binomial distribution, namely $0.99 \cdot \mathrm{Neg\hspace{0.3mm}Bin}\hspace{0.2mm}(r,p) + 0.01 \cdot 100 \cdot \mathrm{Poisson}\hspace{0.2mm}(r p)$, where $p = 1/2$ and $r\in \frac{1}{4} \cdot \{2,3,\dots,60\}$. The bottom row shows the RMSE for the robust estimator, $\hat{p}_n^{\scriptscriptstyle(\mathrm{R})}$, over the ML estimator, $\hat{p}_n^{\scriptscriptstyle(\mathrm{ML})}$.}
            \label{fig:Figure.2.3.analogue.Coeurjolly.polluted}
        \end{figure}

        When $r$ is unknown, we can apply a similar reasoning to obtain a robust estimator $(\hat{r}_n^{(R)}, \hat{p}_n^{(R)})$ for the pair $(r,p)$.
        To do so, we need the asymptotics of another quantile, say the third quartile (other viable choices are also possible).
        For $r > 0$ and $p\in (0,1)$, let $K\sim \mathrm{Neg\hspace{0.3mm}Bin}\hspace{0.2mm}(r, p)$ and $U\sim \text{Uniform}\hspace{0.2mm}(0,1)$.
        By rerunning the proof of Theorem~\ref{thm:median.negative.binomial} up to \eqref{eq:upto}, we want to find $t = \mathrm{ThirdQuartile}\hspace{0.2mm}(K + U) > 0$ such that
        \begin{equation}
            \frac{3}{4} = \frac{1}{2} + \frac{1}{\sqrt{2\pi}} \left[\delta_{\lfloor t \rfloor + 1 - c_{r,p}^{\star}(\lfloor t \rfloor + 1)} \cdot \{t\} + \delta_{\lfloor t \rfloor - c_{r,p}^{\star}(\lfloor t \rfloor)} \cdot (1 - \{t\})\right] + \OO_p(r^{-3/2}).
        \end{equation}
        If we multiply both sides by $ \sqrt{2\pi} \sqrt{r p q^{-2}}$, we get
        \begin{equation}
            \sqrt{\frac{\pi r p}{8 q^2}} = t - \left[c_{r,p}^{\star}(\lfloor t \rfloor + 1) \cdot \{t\} + c_{r,p}^{\star}(\lfloor t \rfloor) \cdot (1 - \{t\})\right] - \frac{r p}{q} + \OO_p(r^{-1}).
        \end{equation}
        After rearranging some terms and applying the expression for $c_{r,p}^{\star}$ in \eqref{eq:optimal.choice.c}, the above is equivalent to
        \begin{equation}\label{eq:thm:third.quartile.negative.binomial}
            t - \frac{r p}{q} - \sqrt{\frac{\pi r p}{8 q^2}} = \frac{1}{2} + \frac{1+p}{6 q} \left[\delta_{\lfloor t \rfloor + \frac{1}{2}}^2 \cdot \{t\} + \delta_{\lfloor t \rfloor - \frac{1}{2}}^2 \cdot \{1 - t\}\right] - \frac{1+p}{6 q} + \OO_p(r^{-1/2}).
        \end{equation}

        \vspace{-1mm}
        Now that we have the asymptotics of both the median and the third quartile (see \eqref{eq:thm:median.negative.binomial} and \eqref{eq:thm:third.quartile.negative.binomial}, respectively), we can generalize the logic from Corollary~\ref{cor:robust.estimator} to create a system of two equations with two unknowns to determine a robust estimator for the pair of parameters $(r,p)$.
        Let $K_1,K_2,\dots,K_n\sim \mathrm{Neg\hspace{0.3mm}Bin}\hspace{0.2mm}(r,p)$ and $U_1,U_2,\dots,U_n\sim \mathrm{Uniform}\hspace{0.2mm}(0,1)$ be i.i.d., and define $X_i \leqdef K_i + U_i$ for all $i\in \{1,2,\dots,n\}$.
        Also, write $\hat{m} = \widehat{\mathrm{Median}}(X_1, X_2, \dots, X_n)$ and $\hat{t} = \widehat{\mathrm{ThirdQuartile}}(X_1, X_2, \dots, X_n)$ for simplicity.
        Define $(\hat{r}_n^{(R)}, \hat{p}_n^{(R)})$ as the pair in $(0,\infty) \times (0,1)$ that solves the following system of equations in the variables $(r,p)$:
        \begin{equation}\label{eq:cor:robust.estimator.two.unknown}
            \begin{aligned}
                &\hat{m} - \frac{r p}{1-p} = \frac{1}{2} - \frac{1 + p}{6 (1-p)}, \\
                &\hat{t} - \frac{r p}{1-p} - \sqrt{\frac{\pi r p}{8 (1-p)^2}} = \frac{1}{2} + \frac{1+p}{6 (1-p)} \left[\delta_{\lfloor \hat{t} \rfloor + \frac{1}{2}}^2 \cdot \{\hat{t}\} + \delta_{\lfloor \hat{t} \rfloor - \frac{1}{2}}^2 \cdot \{1 - \hat{t}\}\right] - \frac{1+p}{6 (1-p)}.
            \end{aligned}
        \end{equation}
        Then $(\hat{r}_n^{(R)}, \hat{p}_n^{(R)}) - (r,p)$ converges in probability to $(0,0)$ as $n\to \infty$ and then $r\to \infty$, using the continuous mapping theorem and the fact that the sample median and sample third quartile converge in probability to their theoretical analogues, respectively.

    \newpage
    \subsection{Le Cam distance bound between negative binomial and normal experiments}\label{sec:LeCam.upper.bound}

        The next theorem (Theorem~\ref{thm:prelim.Carter}) bounds the total variation between a jittered $\mathrm{Neg\hspace{0.3mm}Bin}\hspace{0.2mm}(r, p)$ random variable and a $\mathrm{Normal}\hspace{0.3mm}(r p q^{-1}, r p q^{-2})$ random variable.
        The Le Cam distance bound appears in Theorem~\ref{thm:bound.deficiency.distance} just after.
        A similar approach was used in \cite{MR1922539} and \cite{MR4249129} (who simplified the former proof and improved the rate of convergence) to bound the Le Cam distance between multinomial and multivariate normal experiments.

        For the uninitiated reader, the Le Cam distance is simply the maximum between the deficiencies of the two corresponding experiments, which means that our bound in Theorem~\ref{thm:bound.deficiency.distance} allows us to control the error we make when a statistical conclusion under one model is converted, via an appropriate Markov kernel, to a statistical conclusion under the other model.
        The usefulness of this notion comes from the fact that seemingly completely different statistical experiments can result in asymptotically equivalent inferences using Markov kernels to carry information from one setting to another. For instance, it was famously shown by \citet{MR1425959} that the density estimation problem and the Gaussian white noise problem are asymptotically equivalent in the sense that the Le Cam distance between the two experiments goes to $0$ as the number of observations goes to infinity. The main idea was that the information we get from sampling observations from an unknown density function and counting the observations that fall in the various boxes of a fine partition of the density's support can be encoded using the increments of a properly scaled Brownian motion with drift $t\mapsto \int_0^t \hspace{-1.5mm}\sqrt{f(s)}\hspace{0.5mm}\rd s$, and vice versa.
        An alternative (simpler) proof of this asymptotic equivalence was shown by \citet{MR2102503} who combined a Haar wavelet cascade scheme with coupling inequalities relating the binomial and univariate normal distributions at each step (a similar argument was developed previously by \citet{MR1922539} to derive a multinomial/multivariate coupling inequality).
        Not only \citet{MR2102503} streamlined the proof of the asymptotic equivalence originally shown by \citet{MR1425959}, but their results hold for a larger class of densities and the asymptotic equivalence was also extended to Poisson processes. For an excellent and concise review on Le Cam's theory for the comparison of statistical models, we refer the reader to \cite{MR3850766}.

        \begin{theorem}\label{thm:prelim.Carter}
            For $r > 0$ and $p\in (0,1)$, let $K\sim \mathrm{Neg\hspace{0.3mm}Bin}\hspace{0.2mm}(r, p)$ and $U\sim \mathrm{Uniform}\hspace{0.2mm}(-\tfrac{1}{2},\tfrac{1}{2})$, where $K$ and $U$ are assumed independent.
            Define $X \leqdef K + U$ and let $\widetilde{\PP}_{r,p}$ be the law of $X$.
            In particular, if $\PP_{r,p}$ is the law of $K$, note that
            \begin{equation}
                \widetilde{\PP}_{r,p}(B) = \int_{\N_0} \int_{(-\frac{1}{2},\frac{1}{2})} \ind_{B}(k + u) \rd u \, \PP_{r,p}(\rd k), \quad B\in \mathscr{B}(\R).
            \end{equation}
            Let $\QQ_{r,p}$ be the law of a $\mathrm{Normal}\hspace{0.3mm}(r p q^{-1}, r p q^{-2})$ random variable.
            Then, for $r > 0$ large enough, there exists a constant $C_p > 0$ that depends only on $p$ such that
            \begin{equation}
                \|\widetilde{\PP}_{r,p} - \QQ_{r,p}\| \leq C_p r^{-1/2},
            \end{equation}
            where $\| \cdot \|$ denotes the total variation norm.
        \end{theorem}

        By inverting the Markov kernel that jitters the negative binomial random variable (which consists in rounding off to the nearest integer), we get an upper bound on the Le Cam distance between negative binomial and Gaussian experiments.

        To understand the notation in Theorem~\ref{thm:bound.deficiency.distance} below, the deficiency $\delta(\mathscr{P},\mathscr{Q})$ in \eqref{eq:def:deficiency.one.sided} gives the smallest distance between the untouched normal measure $\QQ_{r,p}$ and the integrated Markov kernel $\int_{\N_0} T_1(k, \cdot \, ) \, \PP_{r,p}(\rd k)$, the latter representing the negative binomial measure $\PP_{r,p}$ after probabilistic manipulations have been applied (using the Markov kernel $T_1$). You can think of $\delta(\mathscr{P},\mathscr{Q})$ as the quantity of information lost when trying to convert a statistical conclusion from the negative binomial model to the corresponding normal model.
        The deficiency $\delta(\mathscr{Q},\mathscr{P})$ has an analogous interpretation although it is important to note that $\delta(\mathscr{P},\mathscr{Q})$ is not equal to $\delta(\mathscr{Q},\mathscr{P})$ in general. The Le Cam distance is then simply defined as the maximum between the two deficiencies, which means, again, that it measures the maximum amount of information that will be lost when communicating between two models/experiments using the optimal probabilistic manipulations (Markov kernels).

        \begin{theorem}[Bound on the Le Cam distance]\label{thm:bound.deficiency.distance}
            Let $r_0 > 0$ and $p\in (0,1)$ be given, and define the experiments
            \vspace{-1mm}
            \begin{alignat*}{6}
                    &\mathscr{P}
                    &&\leqdef &&~\{\PP_{r,p}\}_{r \geq r_0}, \quad &&\PP_{r,p} ~\text{is the measure induced by } \mathrm{Neg\hspace{0.3mm}Bin}\hspace{0.2mm}(r, p), \\[-0.5mm]
                    &\mathscr{Q}\hspace{-0.5mm}
                    &&\leqdef &&~\{\QQ_{r,p}\}_{r \geq r_0}, \quad &&\QQ_{r,p} ~\text{is the measure induced by } \mathrm{Normal}\hspace{0.3mm}(r p q^{-1}, r p q^{-2}).
            \end{alignat*}
            Then, we have the following bound on the Le Cam distance between $\mathscr{P}$ and $\mathscr{Q}$,
            \begin{equation}\label{eq:thm:bound.deficiency.distance.bound}
                \Delta(\mathscr{P},\mathscr{Q}) \leqdef \max\{\delta(\mathscr{P},\mathscr{Q}),\delta(\mathscr{Q},\mathscr{P})\} \leq C_p r_0^{-1/2},
            \end{equation}
            where $C_p > 0$ is a constant that depends only on $p$,
            \begin{equation}\label{eq:def:deficiency.one.sided}
                \begin{aligned}
                    \delta(\mathscr{P},\mathscr{Q})
                    &\leqdef \inf_{T_1} \sup_{r \geq r_0} \Big\|\int_{\N_0} T_1(k, \cdot \, ) \, \PP_{r,p}(\rd k) - \QQ_{r,p}\Big\|, \\
                    \delta(\mathscr{Q},\mathscr{P})
                    &\leqdef \inf_{T_2} \sup_{r \geq r_0} \Big\|\PP_{r,p} - \int_{\R} T_2(z, \cdot \, ) \, \QQ_{r,p}(\rd z)\Big\|, \\
                \end{aligned}
            \end{equation}
            and the infima are taken, respectively, over all Markov kernels $T_1 : \N_0 \times \mathscr{B}(\R) \to [0,1]$ and $T_2 : \R \times \mathscr{B}(\N_0) \to [0,1]$.
        \end{theorem}

        \begin{proof}[Proof of Theorem~\ref{thm:prelim.Carter}]
            Let $X\sim \widetilde{\PP}_{r,p}$.
            By the comparison of the total variation norm with the Hellinger distance on page 726 of \cite{MR1922539}, we already know that
            \begin{equation}\label{eq:first.bound.total.variation}
                \|\widetilde{\PP}_{r,p} - \QQ_{r,p}\| \leq \sqrt{2 \, \PP(X\in B_{r,p}^c(1/2)) + \EE\bigg[\log\bigg(\frac{\rd \widetilde{\PP}_{r,p}}{\rd \QQ_{r,p}}(X)\bigg) \, \ind_{\{X\in B_{r,p}(1/2)\}}\bigg]},
            \end{equation}
            where $B_{r,p}(1/2)$ is defined in \eqref{eq:bulk}.
            By applying a standard large deviation bound, we have, for $r$ large enough,
            \begin{equation}\label{eq:large.deviation.bound}
                \PP(X\in B_{r,p}^c(1/2)) \leq 100 \, \exp\Big(-\frac{r^{1/3}}{100}\Big).
            \end{equation}
            For the expectation in \eqref{eq:first.bound.total.variation}, if $x \mapsto \widetilde{P}_{r,p}(x)$ denotes the density function associated with $\widetilde{\PP}_{r,p}$ (i.e., it is equal to $P_{r,p}(k)$ whenever $k\in \N_0$ is closest to $x$), then
            \begin{align}\label{eq:I.plus.II.plus.III}
                &\EE\bigg[\log\bigg(\frac{\rd \widetilde{\PP}_{r,p}}{\rd \QQ_{r,p}}(X)\bigg) \, \ind_{\{X\in B_{r,p}(1/2)\}}\bigg] \notag \\
                &\quad=\EE\bigg[\log\bigg(\frac{\widetilde{P}_{r,p}(X)}{\frac{q}{\sqrt{r p}} \phi(\delta_{X})}\bigg) \, \ind_{\{X\in B_{r,p}(1/2)\}}\bigg] \notag \\[1mm]
                &\quad= \EE\bigg[\log\bigg(\frac{P_{r,p}(K)}{\frac{q}{\sqrt{r p}} \phi(\delta_{K})}\bigg) \, \ind_{\{K\in B_{r,p}(1/2)\}}\bigg] + \EE\bigg[\log\bigg(\frac{\frac{q}{\sqrt{r p}} \phi(\delta_{K})}{\frac{q}{\sqrt{r p}} \phi(\delta_{X})}\bigg) \, \ind_{\{K\in B_{r,p}(1/2)\}}\bigg] \notag \\[1mm]
                &\quad\quad+ \EE\bigg[\log\bigg(\frac{P_{r,p}(K)}{\frac{q}{\sqrt{r p}} \phi(\delta_{X})}\bigg) \, (\ind_{\{X\in B_{r,p}(1/2)\}} - \ind_{\{K\in B_{r,p}(1/2)\}})\bigg] \notag \\[1mm]
                &\reqdef (\mathrm{I}) + (\mathrm{II}) + (\mathrm{III}).
            \end{align}
            By Lemma~\ref{lem:LLT.negative.binomial},
            \begin{equation}\label{eq:estimate.I.begin}
                \begin{aligned}
                    (\mathrm{I})
                    &= \frac{1}{\sqrt{r p}}  \EE\left[(1 + p) \left(\frac{1}{6} \frac{(K - r p q^{-1})^3}{(r p q^{-2})^{3/2}} - \frac{1}{2} \frac{K - r p q^{-1}}{(r p q^{-2})^{1/2}}\right) \, \ind_{\{K\in B_{r,p}(1/2)\}}\right] \\[1mm]
                    &\quad+ \OO_p\Bigg(\frac{1}{r} \bigg(\frac{\EE[|K - r p q^{-1}|^4]}{(r p q^{-2})^2} + \frac{\EE[|K - r p q^{-1}|^2]}{r p q^{-2}} + 1\bigg)\Bigg) + \OO_p(r^{-3/2}).
                \end{aligned}
            \end{equation}
            By Lemma~\ref{lem:central.moments.negative.binomial}, the first $\OO_p(\cdot)$ term above is $\OO_p(r^{-1})$.
            By Corollary~\ref{cor:central.moments.negative.binomial.on.events}, we can also control the $\asymp_p r^{-1/2}$ above.
            We obtain
            \begin{equation}\label{eq:estimate.I}
                (\mathrm{I}) = \OO_p(r^{-1}) + \OO_p\big(r^{-1/2} (\PP(K\hspace{-0.5mm}\in\hspace{-0.5mm} B_{r,p}^c(1/2)))^{1/2}\big) = \OO_p(r^{-1}).
            \end{equation}
            For the term $(\mathrm{II})$ in \eqref{eq:I.plus.II.plus.III},
            \begin{align}\label{eq:estimate.II.prelim}
                \log\left(\frac{\frac{q}{\sqrt{r p}} \phi(\delta_{K})}{\frac{q}{\sqrt{r p}} \phi(\delta_{X})}\right)
                &= \frac{(X - r p q^{-1})^2}{2 \, r p q^{-2}} - \frac{(K - r p q^{-1})^2}{2 \, r p q^{-2}} \notag \\
                &= \frac{(X - K)^2}{2 \, r p q^{-2}} + \frac{(X - K) (K - r p q^{-1})}{r p q^{-2}}.
            \end{align}
            With our assumption that $K$ and $X - K = U\sim \mathrm{Uniform}\hspace{0.2mm}(-\tfrac{1}{2},\tfrac{1}{2})$ are independent,
            \begin{align}\label{eq:estimate.II}
                (\mathrm{II})
                &= \frac{1/12}{2 \, r p q^{-2}} - \frac{\EE[(X - K)^2 \, \ind_{\{K\in B_{r,p}^c(1/2)\}}]}{2 \, r p q^{-2}} - \frac{\EE[(X - K) (K - r p q^{-1}) \, \ind_{\{K\in B_{r,p}^c(1/2)\}}]}{r p q^{-2}} \notag \\[1.5mm]
                &= \frac{q}{24 \, r p} + \OO\bigg(\frac{\PP(K\hspace{-0.5mm}\in\hspace{-0.5mm} B_{r,p}^c(1/2))}{r p q^{-2}}\bigg) + \OO\Bigg(\frac{\sqrt{\EE[(K - r p q^{-1})^2]} \, \sqrt{\PP(K\hspace{-0.5mm}\in\hspace{-0.5mm} B_{r,p}^c(1/2))}}{r p q^{-2}}\Bigg) \notag \\[2mm]
                &= \OO_p(r^{-1}).
            \end{align}
            For the term $(\mathrm{III})$ in \eqref{eq:I.plus.II.plus.III}, we have the following crude bound from Lemma~\ref{lem:LLT.negative.binomial} and \eqref{eq:estimate.II.prelim}, on the symmetric difference event $\{X\hspace{-0.5mm}\in\hspace{-0.5mm} B_{r,p}(1/2)\} \hspace{0.3mm}\triangle\hspace{0.3mm} \{K\hspace{-0.5mm}\in\hspace{-0.5mm} B_{r,p}(1/2)\}$,\hspace{0.5mm}%
            \footnote{Note that $\{X\hspace{-0.5mm}\in\hspace{-0.5mm} B_{r,p}(1/2)\} \hspace{0.3mm}\triangle\hspace{0.3mm} \{K\hspace{-0.5mm}\in\hspace{-0.5mm} B_{r,p}(1/2)\} \subseteq \{K\hspace{-0.5mm}\in\hspace{-0.5mm} B_{r,p}(3/4)\}$ assuming that $r$ is large enough, simply because $|X - K| \leq \frac{1}{2}$.}
            \begin{align}
                \log\bigg(\frac{P_{r,p}(K)}{\frac{q}{\sqrt{r p}} \phi(\delta_{X})}\bigg)
                &= \log\bigg(\frac{P_{r,p}(K)}{\frac{q}{\sqrt{r p}} \phi(\delta_{K})}\bigg) + \log\bigg(\frac{\frac{q}{\sqrt{r p}} \phi(\delta_{K})}{\frac{q}{\sqrt{r p}} \phi(\delta_{X})}\bigg) \notag \\
                &= \OO_p\bigg(\frac{|K - r p q^{-1}|^3}{(r p q^{-2})^2} + \frac{|K - r p q^{-1}|}{r p q^{-2}} + \frac{1}{r p q^{-2}}\bigg),
            \end{align}
            which yields, by Cauchy-Schwarz and Lemma~\ref{lem:central.moments.negative.binomial},
            \begin{align}\label{eq:estimate.III}
                (\mathrm{III})
                &= \OO_p\left(\hspace{-1mm}
                    \begin{array}{l}
                        \sqrt{\frac{\EE[|K - r p q^{-1}|^6]}{(r p q^{-2})^4} + \frac{\EE[|K - r p q^{-1}|^2] + 1}{(r p q^{-2})^2}} \\[2mm]
                        \cdot \sqrt{\PP\big(\{X\in B_{r,p}(1/2)\} \triangle \{K\in B_{r,p}(1/2)\}\big)}
                    \end{array}
                    \hspace{-1mm}\right) \notag \\[1mm]
                &= \OO_p\left(r^{-1/2} \, \sqrt{\PP\big(\{X\in B_{r,p}(1/2)\} \triangle \{K\in B_{r,p}(1/2)\}\big)}\right).
            \end{align}
            Putting \eqref{eq:estimate.I}, \eqref{eq:estimate.II} and \eqref{eq:estimate.III} in \eqref{eq:I.plus.II.plus.III}, together with the exponential bound
            \begin{align}
                &\PP\big(\{X\in B_{r,p}(1/2)\} \triangle \{K\in B_{r,p}(1/2)\}\big) \notag \\[2.5mm]
                &\qquad\leq \PP(K\hspace{-0.5mm}\in\hspace{-0.5mm} B_{r,p}^c(1/2)) + \PP(X\in B_{r,p}^c(1/2)) \notag \\[-0.5mm]
                &\qquad\leq 2 \cdot 100 \, \exp\Big(-\frac{r^{1/3}}{100}\Big),
            \end{align}
            yields, as $r\to \infty$,
            \begin{equation}\label{eq:I.plus.II.plus.III.end}
                \EE\bigg[\log\bigg(\frac{\rd \widetilde{\PP}_{r,p}}{\rd \QQ_{r,p}}(X)\bigg) \, \ind_{\{X\in B_{r,p}(1/2)\}}\bigg] = (\mathrm{I}) + (\mathrm{II}) + (\mathrm{III}) = \OO_p(r^{-1}).
            \end{equation}
            Now, putting \eqref{eq:large.deviation.bound} and \eqref{eq:I.plus.II.plus.III.end} together in \eqref{eq:first.bound.total.variation} gives the conclusion.
        \end{proof}

        \begin{proof}[Proof of Theorem~\ref{thm:bound.deficiency.distance}]
            By Theorem~\ref{thm:prelim.Carter}, we get the desired bound on $\delta(\mathscr{P},\mathscr{Q})$ by choosing the Markov kernel $T_1^{\star}$ that adds $U$ to $K$, namely
            \begin{equation}
                \begin{aligned}
                    T_1^{\star}(k,B) \leqdef \int_{(-\frac{1}{2},\frac{1}{2})} \ind_{B}(k + u) \rd u, \quad k\in \N_0, ~B\in \mathscr{B}(\R).
                \end{aligned}
            \end{equation}
            To get the bound on $\delta(\mathscr{Q},\mathscr{P})$, it suffices to consider a Markov kernel $T_2^{\star}$ that inverts the effect of $T_1^{\star}$, i.e., rounding off $Z\sim \mathrm{Normal}\hspace{0.3mm}(r p q^{-1}, r p q^{-2})$ to the nearest integer.
            Then, as explained in Section~5 of \cite{MR1922539}, we get
            \begin{align}
                \delta(\mathscr{Q},\mathscr{P})
                &\leq \sup_{r\geq r_0} \Big\|\PP_{r,p} - \int_{\R} T_2^{\star}(z, \cdot \, ) \, \QQ_{r,p}(\rd z)\Big\| \notag \\
                &= \sup_{r\geq r_0} \Big\|\int_{\R} T_2^{\star}(z, \cdot \, ) \int_{\N_0} T_1^{\star}(k, \rd z) \, \PP_{r,p}(\rd k) - \int_{\R} T_2^{\star}(z, \cdot \, ) \, \QQ_{r,p}(\rd z)\Big\| \notag \\
                &\leq \sup_{r\geq r_0} \Big\|\int_{\N_0} T_1^{\star}(k, \cdot \, ) \, \PP_{r,p}(\rd k) - \QQ_{r,p}\Big\|,
            \end{align}
            and we get the same bound by Theorem~\ref{thm:prelim.Carter}.
        \end{proof}

\begin{appendices}

\section{Proof of the refined continuity correction}\label{sec:refined.continuity.correction}

    \begin{proof}[Proof of Lemma~\ref{lem:LLT.negative.binomial}]
        By taking the logarithm in \eqref{eq:negative.binomial.pmf}, we have
        \begin{equation}
            \begin{aligned}
                \log\big(P_{r,p}(k)\big)
                &= \log \Gamma(r + k) - \log \Gamma(r) - \log k! + r \log q + k \log p.
            \end{aligned}
        \end{equation}
        Stirling's formula yields
        \begin{equation}
            \begin{aligned}
                \log \Gamma(x) &= \frac{1}{2} \log (2\pi) + (x - \tfrac{1}{2}) \log x - x + \frac{1}{12 x} + \OO(x^{-3}), \quad x > 0, \\[0.5mm]
                \log k! &= \frac{1}{2} \log (2\pi) + (k + \tfrac{1}{2}) \log k - k + \frac{1}{12 k} + \OO(k^{-3}), \quad k\in \N,
            \end{aligned}
        \end{equation}
        see, e.g., \cite{MR0167642}, p.257.
        Hence, we get
        \begin{align}\label{eq:big.equation}
            \log\big(P_{r,p}(k)\big)
            &= - \frac{1}{2} \log (2\pi) + (r + k) \log (r + k) - r \log (r) - k \log k \notag \\[0.5mm]
            &\quad- \frac{1}{2} \log (r + k) + \frac{1}{2} \log (r) - \frac{1}{2} \log k + r \log q + k \log p \notag \\[0.5mm]
            &\quad+ \frac{1}{12} \big[(r + k)^{-1} - r^{-1} - k^{-1}\big] + \OO\big((r + k)^{-3} + r^{-3} + k^{-3}\big).
        \end{align}
        By writing
        \begin{equation}
            r + k = \frac{r}{q} \Big(1 + \frac{\delta_k}{\sqrt{r p^{-1}}}\Big) \quad \text{and} \quad k = \frac{r p}{q} \Big(1 + \frac{\delta_k}{\sqrt{r p}}\Big),
        \end{equation}
        the above is
        \begin{align}\label{eq:big.equation.next}
            \log\big(P_{r,p}(k)\big)
            &= - \frac{1}{2} \log (2 \pi) - \frac{1}{2} \log \Big(\frac{r p}{q^2}\Big) \notag \\
            &\quad+ \frac{r}{q} \bigg(1 + \frac{\delta_k}{\sqrt{r p^{-1}}}\bigg) \log \bigg(1 + \frac{\delta_k}{\sqrt{r p^{-1}}}\bigg) \notag \\
            &\quad- \frac{r p}{q} \bigg(1 + \frac{\delta_k}{\sqrt{r p}}\bigg) \log \bigg(1 + \frac{\delta_k}{\sqrt{r p}}\bigg) \notag \\
            &\quad- \frac{1}{2} \log \bigg(1 + \frac{\delta_k}{\sqrt{r p^{-1}}}\bigg) - \frac{1}{2} \log \bigg(1 + \frac{\delta_k}{\sqrt{r p}}\bigg) \notag \\
            &\quad+ \frac{q}{12 r} \bigg[\bigg(1 + \frac{\delta_k}{\sqrt{r p^{-1}}}\bigg)^{\hspace{-1mm}-1} \hspace{-1mm}- \frac{1}{q} - \frac{1}{p} \bigg(1 + \frac{\delta_k}{\sqrt{r p}}\bigg)^{\hspace{-1mm}-1}\bigg] \notag \\
            &\quad+ \OO\bigg(\frac{q^3}{r^3} \bigg[\bigg(1 + \frac{\delta_k}{\sqrt{r p^{-1}}}\bigg)^{\hspace{-1mm}-3} \hspace{-1mm}+ \frac{1}{q^3} + \frac{1}{p^3} \bigg(1 + \frac{\delta_k}{\sqrt{r p}}\bigg)^{\hspace{-1mm}-3}\bigg]\bigg).
        \end{align}
        Now, note that for $y \geq \eta - 1$, Lagrange's error bound for Taylor expansions yields
        \begin{equation}
            \begin{aligned}
                (1 + y) \log (1 + y) &= y + \frac{y^2}{2} - \frac{y^3}{6} + \frac{y^4}{12} + \OO\bigg(\frac{y^5}{\eta^4}\bigg), \\
                \log (1 + y) &= y - \frac{y^2}{2} + \OO\bigg(\frac{y^3}{\eta^3}\bigg), \\
                (1 + y)^{-1} &= 1 + \OO\bigg(\frac{y}{\eta^2}\bigg).
            \end{aligned}
        \end{equation}
        By applying these approximations in \eqref{eq:big.equation.next}, we obtain
        \begin{align}\label{eq:big.equation.2}
            &\log\big(P_{r,p}(k)\big) \notag \\
            &= - \frac{1}{2} \log \Big(2 \pi \frac{r p}{q^2}\Big) \notag \\
            &+ \frac{r}{q} \left\{\frac{\delta_k}{\sqrt{r p^{-1}}} + \frac{1}{2} \Big(\frac{\delta_k}{\sqrt{r p^{-1}}}\Big)^2 - \frac{1}{6} \Big(\frac{\delta_k}{\sqrt{r p^{-1}}}\Big)^3 + \frac{1}{12} \Big(\frac{\delta_k}{\sqrt{r p^{-1}}}\Big)^4 + \OO\Big(\frac{1}{\eta^4} \Big(\frac{\delta_k}{\sqrt{r p^{-1}}}\Big)^5\Big)\right\} \notag \\
            &- \frac{r p}{q} \left\{\frac{\delta_k}{\sqrt{r p}} + \frac{1}{2} \Big(\frac{\delta_k}{\sqrt{r p}}\Big)^2 - \frac{1}{6} \Big(\frac{\delta_k}{\sqrt{r p}}\Big)^3 + \frac{1}{12} \Big(\frac{\delta_k}{\sqrt{r p}}\Big)^4 + \OO\Big(\frac{1}{\eta^4} \Big(\frac{\delta_k}{\sqrt{r p}}\Big)^5\Big)\right\} \notag \\
            &- \frac{1}{2} \left\{\frac{\delta_k}{\sqrt{r p^{-1}}} - \frac{1}{2} \bigg(\frac{\delta_k}{\sqrt{r p^{-1}}}\bigg)^2 + \OO\bigg(\frac{1}{\eta^3} \bigg(\frac{\delta_k}{\sqrt{r p^{-1}}}\bigg)^3\bigg)\right\} \notag \\
            &- \frac{1}{2} \left\{\frac{\delta_k}{\sqrt{r p}} - \frac{1}{2} \bigg(\frac{\delta_k}{\sqrt{r p}}\bigg)^2 + \OO\bigg(\frac{1}{\eta^3} \bigg(\frac{\delta_k}{\sqrt{r p}}\bigg)^3\bigg)\right\} \notag \\
            &+ \frac{q}{12 r} \left\{1 - q^{-1} - p^{-1}\right\} + \OO\bigg(\frac{\delta_k}{r^{3/2} \eta^2}\bigg) + \OO_p\bigg(\frac{1}{r^3 \eta^3}\bigg).
        \end{align}
        After some cancellations, we get
        \begin{align}\label{eq:big.equation.3}
            \log\bigg(\frac{P_{r,p}(k)}{\frac{q}{\sqrt{r p}} \phi(\delta_k)}\bigg)
            &= \left\{- \frac{p^2}{6 q \sqrt{p}} \frac{\delta_k^3}{r^{1/2}} + \frac{p^2}{12 q} \frac{\delta_k^4}{r} + \OO_p\bigg(\frac{\delta_k^5}{r^{3/2} \eta^4}\bigg)\right\} \notag \\[-0.5mm]
            &\quad+ \left\{\frac{1}{6 q \sqrt{p}} \frac{\delta_k^3}{r^{1/2}} - \frac{1}{12 p q} \frac{\delta_k^4}{r} + \OO_p\bigg(\frac{\delta_k^5}{r^{3/2} \eta^4}\bigg)\right\} \notag \\[-0.5mm]
            &\quad- \frac{1}{2} \left\{\frac{\delta_k}{\sqrt{r p^{-1}}} - \frac{1}{2} \frac{\delta_k^2}{r p^{-1}} + \OO_p\bigg(\frac{\delta_k^3}{r^{3/2} \eta^3}\bigg)\right\} \notag \\
            &\quad- \frac{1}{2} \left\{\frac{\delta_k}{\sqrt{r p}} - \frac{1}{2} \frac{\delta_k^2}{r p} + \OO_p\bigg(\frac{\delta_k^3}{r^{3/2} \eta^3}\bigg)\right\} \notag \\
            &\quad- \frac{p^2 + q}{12 r p} + \OO_p\bigg(\frac{1 + |\delta_k|}{r^{3/2} \eta^3}\bigg) \notag \\[2mm]
            &= (r p)^{-1/2} \left\{\frac{1 + p}{6} \delta_k^3 - \frac{1 + p}{2} \delta_k\right\} \notag \\
            &\quad+ (r p)^{-1} \left\{- \frac{1 + p + p^2}{12} \delta_k^4 + \frac{p^2 + 1}{4} \delta_k^2 - \frac{p^2 + q}{12}\right\} \notag \\
            &\quad+ \OO_p\bigg(\frac{1 + |\delta_k|^5}{r^{3/2} \eta^4}\bigg),
        \end{align}
        which proves \eqref{eq:lem:LLT.negative.binomial.log.eq}.
        To obtain \eqref{eq:lem:LLT.negative.binomial.eq} and conclude the proof, we take the exponential on both sides of the last equation and we expand the right-hand side with
        \begin{equation}\label{eq:Taylor.exponential}
            e^y = 1 + y + \frac{y^2}{2} + \OO(e^{\widetilde{\eta}} y^3), \quad \text{for } -\infty < y \leq \widetilde{\eta}.
        \end{equation}
        For $r$ large enough and uniformly for $|\delta_k| \leq \eta \, r^{1/6} p^{1/2}$, the right-hand side of \eqref{eq:big.equation.3} is $\OO_p(1)$.
        When this bound is taken as $y$ in \eqref{eq:Taylor.exponential}, it explains the error in \eqref{eq:lem:LLT.negative.binomial.eq}.
    \end{proof}

    \begin{proof}[Proof of Theorem~\ref{thm:main.result}]
        Let $c\in \R$.
        Note that \eqref{eq:thm:main.result.eq.2.reflect} is a trivial consequence of \eqref{eq:thm:main.result.eq.1}, so we only need to prove \eqref{eq:thm:main.result.eq.1}.
        By decomposing $[\delta_{a - c},\infty)$ into small intervals, we get
        \begin{equation}\label{eq:Cressie.generalization.eq.3}
            \begin{aligned}
                \sum_{k=a}^{\infty} P_{r,p}(k) - \int_{\delta_{a - c}}^{\infty} \hspace{-1mm} \phi(y) \rd y
                &~= \hspace{-2mm} \sum_{\substack{k=a \\ k\in B_{r,p}(1/2)}}^{\infty} \hspace{-2mm}\Big[P_{r,p}(k) - \int_{\delta_{a-\frac{1}{2}}}^{\delta_{a+\frac{1}{2}}} \phi(y) \rd y\Big] \\[-0.5mm]
                &~\qquad\qquad- \int_{\delta_{a-c}}^{\delta_{a-\frac{1}{2}}} \phi(y) \rd y + \OO(e^{-\beta r}),
            \end{aligned}
        \end{equation}
        for a small enough constant $\beta = \beta(p) > 0$, where the exponential error comes from the contributions outside of the bulk.
        The Taylor expansion of $\phi(x)$ around any $x_0\in \R$ is
        \begin{equation}\label{eq:Taylor.phi.Sigma}
            \begin{aligned}
                \phi(x)
                &= \phi(x_0) + \phi'(x_0) (x - x_0) + \tfrac{1}{2} \phi''(x_0) (x - x_0)^2 \\[1mm]
                &\quad+ \OO(|x - x_0|^3).
            \end{aligned}
        \end{equation}
        By taking $x_0 = \delta_{k}$ in \eqref{eq:Taylor.phi.Sigma} and integrating on $[\delta_{k-\frac{1}{2}}, \delta_{k+\frac{1}{2}}]$, the first and third order derivatives disappear because of the symmetry.
        We have
        \begin{align}\label{eq:Cressie.generalization.eq.1}
            \int_{\delta_{k-\frac{1}{2}}}^{\delta_{k+\frac{1}{2}}} \phi(y) \rd y
            &= \frac{q}{\sqrt{r p}} \phi(\delta_{k}) + \frac{\phi''(\delta_{k})}{2} \int_{-q/(2\sqrt{r p})}^{q/(2\sqrt{r p})} x^2 \rd x + \OO_p\bigg(\frac{1 + |\delta_k|^5}{r^{5/2}}\bigg) \notag \\
            &= \frac{q}{\sqrt{r p}} \phi(\delta_{k}) \bigg\{1 + \frac{q^2}{24 r p}(\delta_k^2 - 1) + \OO_p\bigg(\frac{1 + |\delta_k|^5}{r^2}\bigg)\bigg\}.
        \end{align}
        Similarly, by taking $x_0 = \delta_{k}$ in \eqref{eq:Taylor.phi.Sigma} and integrating on $[\delta_{a - c}, \delta_{a - \frac{1}{2}}]$, we have
        \begin{equation}\label{eq:Cressie.generalization.eq.2}
            \int_{\delta_{a-c}}^{\delta_{a-\frac{1}{2}}} \phi(y) \rd y = \frac{q}{\sqrt{r p}} \phi(\delta_{a}) \left\{(c - \frac{1}{2}) + \frac{q}{\sqrt{r p}} \frac{1}{2} (c^2 - \frac{1}{4}) + \OO_p\Big(\frac{1 + |\delta_a|^2}{r}\Big)\right\}.
        \end{equation}

        Using \eqref{eq:Cressie.generalization.eq.1}, \eqref{eq:Cressie.generalization.eq.2}, and the expression of $P_{r,p}(k)$ from Lemma~\ref{lem:LLT.negative.binomial} when $k$ is in the bulk, the right-hand side of \eqref{eq:Cressie.generalization.eq.3} is equal to
        \begin{equation}\label{eq:Cressie.generalization.eq.4}
            \begin{aligned}
                &(r p)^{-1/2} \left\{\hspace{-1mm}
                \begin{array}{l}
                    \frac{1 + p}{6} \int_{\{y \geq \delta_{\widetilde{a}}\}} y^3 \phi(y) \rd y - \frac{1 + p}{2} \int_{\{y \geq \delta_{\widetilde{a}}\}} y \, \phi(y) \rd y - q (c - \frac{1}{2}) \phi(\delta_{\widetilde{a}})
                \end{array}
                \hspace{-1mm}\right\} \\[1mm]
                + ~&(r p)^{-1} \left\{\hspace{-1mm}
                \begin{array}{l}
                    \frac{(1 + p)^2}{72 p} \int_{\{y \geq \delta_{\widetilde{a}}\}} y^6 \phi(y) \rd y - \frac{2 + 3 p + 2 p^2}{12} \int_{\{y \geq \delta_{\widetilde{a}}\}} y^4 \phi(y) \rd y \\[1.5mm]
                    + \frac{3 + 2 p + 3 p^2}{8} \int_{\{y \geq \delta_{\widetilde{a}}\}} y^2 \phi(y) \rd y - \frac{p^2 + q}{12} \int_{\{y \geq \delta_{\widetilde{a}}\}} \phi(y) \rd y \\[1mm]
                    - \frac{q^2}{24} \int_{\{y \geq \delta_{\widetilde{a}}\}} (y^2 - 1) \phi(y) \rd y \\[1mm]
                    + \tfrac{q}{2} (c - \frac{1}{2}) \delta_{\widetilde{a}} \phi(\delta_{\widetilde{a}}) - \tfrac{q^2}{2} (c^2 - \frac{1}{4}) \delta_{\widetilde{a}} \phi(\delta_{\widetilde{a}})
                \end{array}
                \hspace{-1mm}\right\} + \OO_p(r^{-3/2}),
            \end{aligned}
        \end{equation}
        where $\widetilde{a} \leqdef a - \tfrac{1}{2}$.
        For $d\in \R$, consider
        \begin{equation}\label{eq:error.terms.brace.2}
            c = \frac{1}{2} + \bigg[\frac{1 + p}{6 q} \cdot \frac{\int_{\{y \geq \delta_{\widetilde{a}}\}} y^3 \phi(y) \rd y}{\phi(\delta_{\widetilde{a}})} - \frac{1 + p}{2 q} \cdot \frac{\int_{\{y \geq \delta_{\widetilde{a}}\}} y \, \phi(y) \rd y}{\phi(\delta_{\widetilde{a}})}\bigg] + \frac{d}{q \sqrt{r p}},
        \end{equation}
        in \eqref{eq:Cressie.generalization.eq.4}.
        The terms of order $(r p)^{-1/2}$ cancel out and the $d$ that cancels the terms of order $(r p)^{-1}$ is
        \begin{align}\label{eq:error.terms.brace.2.next}
            d_{r,p}^{\star}(a) =
            \left\{\hspace{-1mm}
            \begin{array}{l}
                \frac{(1 + p)^2}{72} \int_{\{y \geq \delta_{\widetilde{a}}\}} y^6 \phi(y) \rd y - \frac{2 + 3 p + 2 p^2}{12} \int_{\{y \geq \delta_{\widetilde{a}}\}} y^4 \phi(y) \rd y \\[1.5mm]
                + \frac{3 + 2 p + 3 p^2}{8} \int_{\{y \geq \delta_{\widetilde{a}}\}} y^2 \phi(y) \rd y - \frac{p^2 + q}{12} \int_{\{y \geq \delta_{\widetilde{a}}\}} \phi(y) \rd y \\[1mm]
                - \frac{q^2}{24} \int_{\{y \geq \delta_{\widetilde{a}}\}} (y^2 - 1) \phi(y) \rd y \\[1mm]
                + \frac{p}{2} \Big[\frac{1 + p}{6} \cdot \frac{\int_{\{y \geq \delta_{\widetilde{a}}\}} y^3 \phi(y) \rd y}{\phi(\delta_{\widetilde{a}})} - \frac{1 + p}{2} \cdot \frac{\int_{\{y \geq \delta_{\widetilde{a}}\}} y \, \phi(y) \rd y}{\phi(\delta_{\widetilde{a}})}\Big] \delta_{\widetilde{a}} \phi(\delta_{\widetilde{a}}) \\
                - \tfrac{1}{2} \Big[\frac{1 + p}{6} \cdot \frac{\int_{\{y \geq \delta_{\widetilde{a}}\}} y^3 \phi(y) \rd y}{\phi(\delta_{\widetilde{a}})} - \frac{1 + p}{2} \cdot \frac{\int_{\{y \geq \delta_{\widetilde{a}}\}} y \, \phi(y) \rd y}{\phi(\delta_{\widetilde{a}})}\Big]^2 \delta_{\widetilde{a}} \phi(\delta_{\widetilde{a}})
            \end{array}
            \hspace{-1mm}\right\} \frac{1}{\phi(\delta_{\widetilde{a}})}.
        \end{align}
        Now, using the fact that, for $a\in \R$,
        \begin{equation}
            \begin{aligned}
                &\int_{\{y \geq \delta_{\widetilde{a}}\}} y^6 \phi(y) \rd y = (15 \delta_{\widetilde{a}} + 5 \delta_{\widetilde{a}}^3 + \delta_{\widetilde{a}}^5) \phi(\delta_{\widetilde{a}}) + 15 \Psi(\delta_{\widetilde{a}}), \\
                &\int_{\{y \geq \delta_{\widetilde{a}}\}} y^4 \phi(y) \rd y = (3 \delta_{\widetilde{a}} + \delta_{\widetilde{a}}^3) \phi(\delta_{\widetilde{a}}) + 3 \Psi(\delta_{\widetilde{a}}), \\
                &\int_{\{y \geq \delta_{\widetilde{a}}\}} y^3 \phi(y) \rd y = (2 + \delta_{\widetilde{a}}^2) \phi(\delta_{\widetilde{a}}), \\
                &\int_{\{y \geq \delta_{\widetilde{a}}\}} y^2 \phi(y) \rd y = \delta_{\widetilde{a}} \phi(\delta_{\widetilde{a}}) + \Psi(\delta_{\widetilde{a}}), \\
                &\int_{\{y \geq \delta_{\widetilde{a}}\}} y \phi(y) \rd y = \phi(\delta_{\widetilde{a}}),
            \end{aligned}
        \end{equation}
        where $\Psi$ denotes the survival function of the standard normal distribution, the $c$ that cancel both braces in \eqref{eq:Cressie.generalization.eq.4} is
        \begin{equation}
            \begin{aligned}
                c_{r,p}^{\star}(a)
                &= \frac{1}{2} + \bigg[\frac{1 + p}{6 q} \cdot \frac{\int_{\{y \geq \delta_{\widetilde{a}}\}} y^3 \phi(y) \rd y}{\phi(\delta_{\widetilde{a}})} - \frac{1 + p}{2 q} \cdot \frac{\int_{\{y \geq \delta_{\widetilde{a}}\}} y \, \phi(y) \rd y}{\phi(\delta_{\widetilde{a}})}\bigg] + \frac{d_{r,p}^{\star}(a)}{q \sqrt{r p}} \\[1mm]
                &= \frac{1}{2} + \frac{1+p}{6 q} \big[\delta_{\widetilde{a}}^2 - 1\big] + \frac{1}{q \sqrt{r p}} \left\{\hspace{-1mm}
                    \begin{array}{l}
                        \left[\hspace{-1mm}
                            \begin{array}{l}
                                \frac{(1 + p)^2}{72} \cdot 5 - \frac{2 + 3 p + 2 p^2}{12} \\
                                + \frac{p(1 + p)}{12} \cdot (1 - 3) \\
                                + \frac{(1 + p)^2}{72} \cdot (-4 + 6)
                            \end{array}
                            \hspace{-1mm}\right] \delta_{\widetilde{a}}^3 \\[7mm]
                        \left[\hspace{-1mm}
                            \begin{array}{l}
                                \frac{(1 + p)^2}{72} \cdot 15 - \frac{2 + 3 p + 2 p^2}{12} \cdot 3 \\
                                + \frac{3 + 2 p + 3 p^2}{8} - \frac{q^2}{24} \\
                                + \frac{p(1 + p)}{12} \cdot (2 - 3) \\
                                + \frac{(1 + p)^2}{72} \cdot (-4 + 6 \cdot 2 - 9)
                            \end{array}
                            \hspace{-1mm}\right] \delta_{\widetilde{a}}
                    \end{array}
                    \hspace{-1mm}\right\} \\[1mm]
                &= \frac{1}{2} + \frac{1+p}{6 q} \big[\delta_{\widetilde{a}}^2 - 1\big] + \frac{1}{q \sqrt{r p}} \left\{\hspace{-1mm}
                    \begin{array}{l}
                        - \frac{1}{72} \big[5 + 16 p + 17 p^2\big] \delta_{\widetilde{a}}^3 \\[1mm]
                        + \frac{1}{36} \big[1 - 4 p - 2 p^2\big] \delta_{\widetilde{a}}
                    \end{array}
                    \hspace{-1mm}\right\}.
            \end{aligned}
        \end{equation}
        This ends the proof.
    \end{proof}

\section{Moments of the negative binomial distribution}\label{sec:negative.binomial.moments}

    In the lemma below, we compute the second, third, fourth and sixth central moments. It is used to control some expectations in \eqref{eq:estimate.III} and the $\asymp_p r^{-1}$ errors in \eqref{eq:estimate.I.begin} of the proof of Theorem~\ref{thm:prelim.Carter}.
    It is also a preliminary result for the proof of Corollary~\ref{cor:central.moments.negative.binomial.on.events} below, where the central moments are bounded on various events.

    \begin{lemma}[Central moments]\label{lem:central.moments.negative.binomial}
        Let $K\sim \mathrm{Neg\hspace{0.3mm}Bin}\hspace{0.2mm}(r, p)$ for some $r > 0$ and $p\in (0,1)$.
        We have
        \vspace{0.5mm}
        \begin{equation}\label{eq:lem:central.moments.negative.binomial.examples}
            \begin{aligned}
                &\EE[(K - r p q^{-1})^2] = r p q^{-2}, \\[0.5mm]
                &\EE[(K - r p q^{-1})^3] = r p q^{-2} \cdot (1 + p) q^{-1}, \\[0.5mm]
                &\EE[(K - r p q^{-1})^4] = 3 (r p q^{-2})^2 + \OO_p(r), \\[0.5mm]
                &\EE[(K - r p q^{-1})^6] = 15 (r p q^{-2})^3 + \OO_p(r^2).
            \end{aligned}
        \end{equation}
    \end{lemma}

    \begin{proof}[Proof of Lemma~\ref{lem:central.moments.negative.binomial}]
        This was calculated using \texttt{Mathematica}.
    \end{proof}

    Next, we bound the first and third central moments on various events.
    The corollary below is used to control the $\asymp_p r^{-1/2}$ errors in \eqref{eq:estimate.I.begin} of the proof of Theorem~\ref{thm:prelim.Carter}.

    \begin{corollary}\label{cor:central.moments.negative.binomial.on.events}
        Let $K\sim \mathrm{Neg\hspace{0.3mm}Bin}\hspace{0.2mm}(r, p)$ for some $r > 0$ and $p\in (0,1)$, and let $A\in \mathscr{B}(\R)$ be a Borel set.
        Then,
        \begin{equation}\label{eq:cor:central.moments.negative.binomial.on.events.main}
            \begin{aligned}
                &\EE[(K - r p q^{-1}) \, \ind_{\{K\in A\}}] = \OO_p(r^{1/2}) (\PP(K\in A^c))^{1/2}, \\[0.5mm]
                &\EE[(K - r p q^{-1})^3 \, \ind_{\{K\in A\}}] = r p q^{-2} \cdot (1 + p) q^{-1} + \OO_p(r^{3/2}) \, (\PP(K\in A^c))^{1/2}.
            \end{aligned}
        \end{equation}
    \end{corollary}

    \begin{proof}[Proof of Corollary~\ref{cor:central.moments.negative.binomial.on.events}]
        This follows from Lemma~\ref{lem:central.moments.negative.binomial} and Holder's inequality.
    \end{proof}

\section{Short proof for the asymptotics of the median of a jittered Poisson random variable}\label{sec:jittered.Poisson.median}

    In this section, we present a short proof for the asymptotics of the median of a Poisson random variable jittered by a uniform (Theorem~\ref{thm:median.Poisson}), using the same technique introduced in Section~\ref{sec:median.jittered}.
    Our statement is slightly weaker than Theorem~1 in \cite{MR4135709}, but the proof is conceptually simpler.

    \begin{theorem}\label{thm:median.Poisson}
        Let $N_{\lambda}\sim \mathrm{Poisson}\hspace{0.2mm}(\lambda)$ and $U\sim \text{Uniform}\hspace{0.2mm}(0,1)$.
        Then, as $\lambda\to \infty$, we have
        \begin{equation}\label{eq:thm:median.Poisson}
            \mathrm{Median}\hspace{0.2mm}(N_{\lambda} + U) - \lambda = \frac{1}{3} + \OO(\lambda^{-1}).
        \end{equation}
    \end{theorem}

    \begin{proof}
            By conditioning on $U$ and using the local limit theorem from Lemma~2.1 in \cite{MR4213687}, we want to find $t = \mathrm{Median}\hspace{0.2mm}(N_{\lambda} + U) > 0$ such that
            \begin{align}\label{eq:thm:median.negative.binomial.beginning.Poisson}
                \frac{1}{2}
                &= \int_0^1 \PP(N_{\lambda} \leq t - u) \, \rd u \notag \\[1mm]
                &= \PP(N_{\lambda} \leq \lfloor t \rfloor) \cdot \{t\} + \PP(N_{\lambda} \leq \lfloor t \rfloor - 1) \cdot (1 - \{t\}) \notag \\[1mm]
                &= \Phi(\delta_{\lfloor t \rfloor + 1 - c_{\lambda}^{\star}(\lfloor t \rfloor + 1)}) \cdot \{t\} + \Phi(\delta_{\lfloor t \rfloor - c_{\lambda}^{\star}(\lfloor t \rfloor)}) \cdot (1 - \{t\}) + \OO(\lambda^{-3/2}),
            \end{align}
            where $\{t\}$ denotes the fractional part of $t$, and
            \begin{equation}\label{eq:optimal.c.Poisson}
                c_{\lambda}^{\star}(a) = \frac{1}{2} + \frac{1}{6} \big[\lambda^{-1} (a - \lambda)^2 - 1\big] + \OO(\lambda^{-1}), \quad \text{for } a - \lambda \asymp 1.
            \end{equation}
            Now, we have the following Taylor series expansion for $\Phi$ at $0$:
            \begin{equation}
                \Phi(x) = \frac{1}{2} + \frac{x}{\sqrt{2\pi}} + \OO(x^3).
            \end{equation}
            Therefore, \eqref{eq:thm:median.negative.binomial.beginning.Poisson} becomes
            \begin{equation}
                0 = \frac{1}{\sqrt{2\pi}} \left[\delta_{\lfloor t \rfloor + 1 - c_{\lambda}^{\star}(\lfloor t \rfloor + 1)} \cdot \{t\} + \delta_{\lfloor t \rfloor - c_{\lambda}^{\star}(\lfloor t \rfloor)} \cdot (1 - \{t\})\right] + \OO(\lambda^{-3/2}).
            \end{equation}
            After rearranging some terms, this is equivalent to
            \begin{equation}
                 t - \lambda = c_{\lambda}^{\star}(\lfloor t \rfloor + 1) \cdot \{t\} + c_{\lambda}^{\star}(\lfloor t \rfloor) \cdot (1 - \{t\}) + \OO(\lambda^{-1}).
            \end{equation}
            By applying the expression for $c_{\lambda}^{\star}$ in \eqref{eq:optimal.c.Poisson}, this is $t - \lambda = \frac{1}{2} - \frac{1}{6} + \OO(\lambda^{-1})$.
        \end{proof}

    Let $N_1,N_2,\dots,N_n\sim \mathrm{Poisson}\hspace{0.2mm}(\lambda)$ and $U_1,U_2,\dots,U_n\sim \mathrm{Uniform}\hspace{0.2mm}(0,1)$ be i.i.d., and define $Z_i \leqdef N_i + U_i$ for all $i\in \{1,2,\dots,n\}$.
    Then, we have the convergence in probability of the sample median:
    \begin{equation}
        \widehat{\mathrm{Median}}(Z_1, Z_2, \dots, Z_n) \stackrel{\PP}{\longrightarrow} \mathrm{Median}\hspace{0.2mm}(Z_1), \quad \text{as } n\to \infty,
    \end{equation}
    see, e.g., \cite{MR1652247}, p.47.
    We deduce the following corollary.

    \begin{corollary}\label{cor:robust.estimator.Poisson}
        With the above notation,
        \begin{equation}\label{eq:cor:robust.estimator.Poisson}
            \widehat{\mathrm{Median}}(Z_1, Z_2, \dots, Z_n) - \lambda \stackrel{\PP}{\longrightarrow} \frac{1}{3},
        \end{equation}
        as $n\to \infty$ and then $\lambda\to \infty$.
    \end{corollary}

\end{appendices}

\section*{Supplemental material}

The \texttt{R} codes that generated Figures~\ref{fig:Figure.1.analogue.Coeurjolly},~\ref{fig:Figure.2.3.analogue.Coeurjolly}~and~\ref{fig:Figure.2.3.analogue.Coeurjolly.polluted} are available at \\ \href{https://www.dropbox.com/sh/6dahqrbtvgxazg7/AAA_SUramLoTC0Eg-HKIO74Ja?dl=0}{https://www.dropbox.com/sh/6dahqrbtvgxazg7/AAA\_SUramLoTC0Eg-HKIO74Ja?dl=0}.

\section*{Conflict of interest statement}

The author declares no conflict of interest.

\section*{Funding}

The author was previously supported by a postdoctoral fellowship from the NSERC (PDF) and the FRQNT (B3X supplement). The author is currently supported by a postdoctoral fellowship (CRM-Simons) from the Centre de recherches math\'ematiques (Universit\'e de Montr\'eal) and the Simons Foundation.

\section*{Acknowledgments}

We thank the referee for carefully reading the manuscript and for his/her helpful comments and suggestions which led to improvements in the writing of this paper.

%
%

\section*{References}
\phantomsection
\addcontentsline{toc}{chapter}{References}

\bibliographystyle{authordate1}
\bibliography{Ouimet_2023_negative_binomial_median_bib}

\end{document}